\documentclass{amsart}
\usepackage{amscd,amssymb,amsopn,amsmath,amsthm,graphics,amsfonts,enumerate,verbatim,calc}
\usepackage[dvips]{graphicx}
\usepackage[colorlinks=true,linkcolor=red,citecolor=blue]{hyperref}
\input xy
\xyoption{all}
 \calclayout
\theoremstyle{plain}
\newtheorem{theorem}{Theorem}[section]
\newtheorem{corollary}[theorem]{Corollary}
\newtheorem{lemma}[theorem]{Lemma}
\newtheorem{proposition}[theorem]{Proposition}

\theoremstyle{definition}

\newtheorem{definition}[theorem]{Definition}
\newtheorem{example}[theorem]{Example}
\newtheorem{remark}[theorem]{Remark}
\newcommand{\begintheorem}{\addtocounter{equation}{1}\begin{theorem}}
\newcommand{\beginlemma}{\addtocounter{equation}{1}\begin{lemma}}
\newcommand{\beginproposition}{\addtocounter{equation}{1}\begin{proposition}}
\newcommand{\begindefinition}{\addtocounter{equation}{1}\begin{definition}}
\newcommand{\begincorollary}{\addtocounter{equation}{1}\begin{corollary}}

\begin{document}
\title{ On pseudo weakly compact operators of  order $ P$}
\author{M. Alikhani}
\address{$^{a}$Department of Mathematics, University of Isfahan, Isfahan, Iran \newline}
\thanks{
 e-mail: {\sf m2020alikhani@yahoo.com}, 
   } \address[]{}

\begin{abstract} 
In this paper, we introduce the concept of a pseudo weakly compact operator of order $ p $ 
between  Banach spaces.\ Also we study  the notion of  $ p $-Dunford-Pettis relatively compact property   which is in ``general"  weaker than the
Dunford-Pettis relatively compact property and 
gives some characterizations of  Banach spaces which have this property.\ Moreover,  by
using  the notion of  $ p $-Right subsets of a dual Banach space, we study the concepts of $ p $-sequentially Right  and weak $ p $-sequentially Right properties   on Banach spaces.\ Furthermore, we obtain some suitable conditions on Banach spaces  $  X$ and
$ Y $ such that projective tensor and injective tensor products between $ X $ and $ Y $ have the $ p $-sequentially Right  property.\ Finally, we introduce  two  properties for the Banach spaces, namely  $ p $-sequentially  Right$ ^{\ast} $ and weak $ p $-sequentially  Right$ ^{\ast} $ properties and obtain some characterizations of these properties.\
\end{abstract}
\maketitle
{\bf keyword}: Dunford-Pettis relatively compact property, pseudo weakly compact operators, sequentially Right property.\
\section{ Introduction}
The notion of Dunford-Pettis set was introduced by Andrews \cite{An} as follows: ``
A bounded subset $ K $ of a Banach space $ X$ is called Dunford-Pettis, if  every weakly compact operator from $ X $ to an arbitrary Banach space $ Y$ maps $ K$ onto a relatively norm compact set of  $ Y $''.\ Equivalently, a bounded 
subset $ K$ of a Banach space $ X $ is a Dunford-Pettis set if and only if
every weakly null sequence $(x^{\ast}_{n})_n $ in $ X^{\ast}, $ converges uniformly to zero on  the set $K;$ that is,
$$\displaystyle\lim_{n\rightarrow\infty}\sup_{x\in K}\vert x^{\ast}_{n}( x)\vert=0,$$\
where $ X^{\ast} $ is a dual space of $ X. $\
It is clear that the class of  Dunford-Pettis sets  strictly contains  the class of relatively compact sets.\ But in general the converse is not true.\ For example, the closed unit ball $ B_{c_{0}} $ is a Dunford-Pettis set in $ c_{0} ,$ while it is not relatively compact.\
The concept of  Dunford-Pettis relatively compact property on Banach spaces presented by Emmanuele \cite{e1}
 as follows: `` A Banach space $ X $ has the Dunford-Pettis relatively compact property (in short $ X\in(DPrcP) $), if every Dunford-Pettis subset of $ X$ is relatively compact".\ For instance, reflexive spaces and Schur spaces (i.e., weak and norm convergence of sequences in $ X $ are coincide) have the $ (DPrcP). $\
Recently, Wen and Chen \cite{w} introduced the concept of Dunford-Pettis completely continuous operators between two arbitrary Banach spaces $ X $ and $ Y$ and obtained some properties of this concept related to some well-known classes of
operators and especially, related to the  Dunford-Pettis relatively compact  property 
of a space $  X$ or $ Y. $\ A bounded linear operator $ T:X\rightarrow Y $ is Dunford-Pettis completely continuous, if carry
Dunford-Pettis and weakly null sequences to norm null ones.\ The class of  Dunford-Pettis completely continuous operators from $ X $ to $ Y$ is denoted
by $ DPcc(X, Y ) .$\\
Peralta et al. \cite{pvwy}, proved
that for a given Banach space $ X $ there is a locally convex topology on $ X, $ which is
called  the`` Right topology ", such that a linear map $ T$ from $ X $ into
a Banach space $ Y$ is weakly compact if and only if it is  Right-to-norm sequentially continuous.\ 
Also, they introduced
the concepts of pseudo weakly compact operators and sequentially Right property on Banach spaces as follows:
\begin{itemize}
\item A  bounded linear $ T $ from a Banach space $ X $ to a Banach space $ Y $ is called pseudo weakly compact, if it transforms Right-null sequences into norm-null sequences; in the other words, when $ (x_{n})_{n}\subset X $ in the Right topology converge to zero, then $ \Vert T(x_{n})  \Vert\rightarrow 0. $\ The class of  pseudo weakly compact operators is denoted by $ PwC(X, Y ) .$\
\item A Banach space $ X $ has the sequentially Right property (in short $ X\in (SR) $), if every pseudo weakly compact
operator $ T $ from $ X $ to a Banach space $ Y $ is weakly compact.
\end{itemize} 

Later on  Kacena \cite{ka} by introducing the notion of  Right set in $ X^{\ast}, $  showed that a Banach space $ X $ has the sequentially Right property if and only if every Right subset of $ X^{\ast} $ is relatively weakly compact.\
A bounded subset $ K $ of $ X^{\ast} $ is a Right set, if every Right-null sequence $ (x_{n})_{n}$ in $ X $ converges uniformly to zero on $ K;$ that is,
$$\displaystyle\lim_{n\rightarrow\infty}\sup_{x^{\ast}\in K}\vert x^{\ast}( x_{n})\vert=0.$$\
Retbi and Wahbi \cite{rw}, introduced
the concepts of $( L) $-Dunford-Pettis sets and $ (L)$-Dunford-Pettis property as follows:\
A bounded subset $ K $ of $ X^{\ast} $ is called an $(L)$-Dunford-Pettis set, if every weakly null sequence $ (x_{n})_{n} $ whose the corresponding set of its a Dunford-Pettis set in $ X $ converges uniformly to zero on $ K;$ that is,
$$\displaystyle\lim_{n\rightarrow\infty}\sup_{x^{\ast}\in K}\vert x^{\ast}( x_{n})\vert=0.$$\
A Banach space $ X $ has the $( L )$-Dunford-Pettis property, if every $ (L) $-Dunford-Pettis subset of $ X^{\ast} $ is relatively weakly compact.\ Recently, Cilia and  Emmanuele in \cite{ce1} and  Ghenciu in \cite{g8}
 obtained a  characterization for Right null sequences.\ In fact, they showed that a sequence $ (x_{n})_{n} $ in a Banach space $  X$ is Right null if and only if it is Dunford-Pettis and weakly null.\
Hence, 
from the above observations,  we get the following results:\
\begin{itemize}
\item A bounded subset $ K $ of  $ X^{\ast} $ is a  Right set   if and only if it is an $(L)$-Dunford-Pettis set.\
\item The class of  Dunford-Pettis completely continuous operators and the class of pseudo weakly compact
operators between Banach spaces coincide.\
\item A Banach space  $ X $ has the $ (SR) $ property if and only if it has the $( L) $-Dunford-Pettis property.\
\end{itemize}
Recently, Ghenciu  \cite{g9} introduced
the concepts  of Dunford-Pettis $  p$-convergent operators, $ p $-Dunford-Pettis relatively compact property (in short $ p $-$ (DPrcP) $), $ p $-Right sets and $ p $-sequentially Right property ( in short $ p $-$ (SR) $) as follows:
\begin{itemize}
\item  A bounded linear operator $T:X\rightarrow Y $ is called Dunford-Pettis $  p$-convergent, if it takes Dunford-Pettis weakly $ p $-summable sequences to norm null sequences.\ 
\item  A Banach space $ X$ has the $ p $-Dunford-Pettis relatively compact property (in short $ p $-$ (DPrcP) $), if every  Dunford-Pettis weakly p-summable sequence $ (x_{n})_{n} $ in $ X $ is norm null.\ 
\item A  bounded subset $ K $ of the dual space $ X^{\ast} $ is called a $ p$-Right set, if every Dunford-Pettis weakly $ p$-summable sequence $ (x_{n})_{n} $ in $ X$ converges uniformly to zero on $ K,$ that is, 
$$\displaystyle\lim_{n} \displaystyle\sup_{x^{\ast}\in K}\vert  x^{\ast}( x_{n}) \vert=0.$$\ 
\item A  Banach space  $ X $ has the $ p $-sequentially Right property  ( in short $ p$-$(SR) $), if every $ p $-Right set  in $ X^{\ast}$ is relatively weakly compact.\
\end{itemize}
 Motivated by the above works, 
in Section 2, we introduce the concept of pseudo weakly compact
operators of order $ p, $  which is the extension of pseudo weakly compact
operators and obtain some  characterizations  of these operators.\ 
Then, we study the notion of $ p $-$ (DPrcP) $ which is in “general” weaker than
the Dunford-Pettis relatively compact property and give some characterizations of  Banach spaces which
have this property.\ Also,
it is proved that
under which conditions  the class of  $ p $-convergent operators and the class of  pseudo weakly compact operators of order $ p $
between Banach spaces coincide.\ 

Section 3 is concerned with $ p $-sequentially Right property, which is a generalization of the sequentially Right property.\
In this section, we answer to the following question: ``Under which conditions the class of weakly compact  operators and the class of pseudo
weakly compact operators of order $ p $ between Banach spaces coincide".\
In addition, we investigate the stability of $ p $-sequentially Right property for some subspaces of  bounded linear operators,  projective tensor product and injective tensor product between Banach spaces $ X$ and $Y.$\
Finally, by introducing the notion of weak $ p $-sequentially Right  property, we give an operator characterization from the class of  $ p $-Right sets which
are  weakly precompact.\ 
In the last  section of the present paper, 
motivated by the notions of $ p $-Right sets and
$  p$-sequentially Right property,
 we introduce the  concepts of
$ p $-Right$ ^{\ast} $ sets, $ p $-sequentially Right$ ^{\ast} $ property and weak
$ p $-sequentially Right$ ^{\ast} $ property on Banach spaces and obtain some characterizations of these properties.

In what follows we introduce some notation and notions which will be used
in the sequel.\ Throughout this paper  $ X,Y $ and $  Z$ are arbitrary Banach spaces and $ 1\leq p\leq \infty. $\ We suppose  $p^{\ast}$ is the H$\ddot{\mathrm{o}}$lder conjugate of $p;$ if $ p=1,~~ \ell_{p^{\ast}} $
 plays the role of $ c_{0} .$\ The unit coordinate vector in $ \ell_{p} $ (resp.\ $ c_{0} $ or  $\ell_{\infty} $) is denoted by $ e_{n}^{p} $ (resp.\ $ e_{n} $).\
The space $ X $ embeds in $ Y, $ if $ X $ is isomorphic to a closed subspace of $  Y$ (in short we denote $ X\hookrightarrow Y $).\ We denote two isometrically isomorphic spaces $ X $ and $ Y $ by
$ X\cong Y.$\ Also we use $ \langle  x,x^{\ast}\rangle $ or $ x^{\ast}(x) $
for the duality between $ x\in X $ and $ x^{\ast}\in X^{\ast}. $\ For a bounded linear operator $ T : X \rightarrow Y, $ the adjoint of the operator $ T $
is denoted by $ T^{\ast}. $\  The space of all  bounded linear operators, weakly
compact operators, and compact operators from $  X$ to $ Y $ will be denoted by
$ L(X, Y ), W(X, Y ),  $ and $ K(X, Y ) ,$ respectively.\
We refer the reader for undefined terminologies to the
classical references \cite{AlbKal, di,du, di1,djt}.\ 
A  sequence $ (x_{n})_{n} $ in $ X $ is called weakly $ p $-summable, if $ (x^{\ast}(x_{n}))_{n} \in \ell_{p}$ for each $ x^{\ast}\in X^{\ast} $ \cite{djt}.\ We denote the set of all weakly $ p $-summable sequences in $ X $ by $ \ell_{p} ^{w}(X).$\
A sequence $(x_{n})_{n}$ in $ X $ is said to be weakly $ p $-convergent to
$ x\in X$ if $ (x_{n} - x)_{n}\in \ell_{p} ^{w}(X).$\ The concept of weakly $  p$-Cauchy sequence 
introduced by Chen et al. \cite{ccl}.\ A sequence $(x_{n})_{n}$ in a Banach space $ X $
is weakly $  p$-Cauchy if for each pair of strictly increasing sequences $(k_{n})_{n}$ and $(j_{n})_{n}$ of positive integers, the sequence $(x_{k_{n}}- x_{j_{n}})_{n}$ is weakly $ p $-summable in $ X .$\ Notice that, every weakly $ p $-convergent sequence is weakly $ p $-Cauchy, and the weakly $ \infty $-Cauchy sequences are precisely the weakly Cauchy sequences.\ 
In what follows we give some concepts which will be used
in the sequel:\
\begin{itemize}
\item A subset $ K $ of a Banach space $ X $ is called relatively weakly $p$-compact, if each sequence in $ K $ admits a weakly $ p$-convergent subsequence with limit in $ X .$\ If the ``limit point” of each weakly $ p $-convergent subsequence lies in $ K, $ then we say that $ K $ is a weakly p-compact set \cite{cs}.\ 
\item A bounded subset $ K $ of $ X^{\ast} $ is a $ p $-$ (V ) $ set, if $ \displaystyle \lim_{n\rightarrow\infty}\displaystyle \sup_{x^{\ast}\in K}\vert x^{\ast}(x_{n}) \vert =0,$
for every weakly $ p $-summable sequence $ (x_{n})_{n} $ in $ X $ \cite{ccl1}.\
\item A Banach space $ X $ has
Pelczy\'{n}ski's  property $ (V) $ of order $  p$ ( $
p$-$ (V) $ property), if every $ p $-$ (V) $ subset of $
X^{\ast}$ is relatively weakly compact \cite{ccl1}.\
\item A bounded subset $ K $ of $ X $ is a $ p $-$ (V^{\ast} ) $ set, if
$  \lim_{n\rightarrow\infty} \sup_{x\in K}\vert x^{\ast}_{n}(x) \vert =0,$
for every weakly $ p $-summable sequence $ (x^{\ast}_{n})_{n} $ in $ X^{\ast}$ \cite{ccl1}.\
\item A bounded linear operator $ T : X \rightarrow  Y  $ is called weakly $ p $-compact, if $ T (B_{X}) $ is a relatively weakly $ p $-compact set in $ Y$ \cite{cs}.\
\item A bounded linear operator $ T : X \rightarrow Y $ is called  weakly limited, if $ T(B_{X}) $ is a Dunford-Pettis set in $ Y $ \cite{w}.\ 
\item  A subset $ K $ of a Banach space $ X $ is called weakly $ p $-precompact,
if every sequence from $  K$ has a weakly $  p$-Cauchy subsequence.\ The weakly $ \infty $-precompact sets are precisely the weakly precompact sets \cite{ccl}.
\item A bounded linear operator $T:X\rightarrow Y$ called $ p$-convergent, if it transforms any weakly $ p $-summable sequence into norm-null sequence.\ The $  \infty$-convergent operators are precisely the completely continuous operators \cite{cs}.
\item A Banach space $ X $ has the $ p $-Schur property (in short $ X\in(C_{p}) $), if every weakly $ p$-summable sequence in $ X $ is norm null.\ It is  clear that, $ X $ has the $ \infty $-Schur property if and only if every weakly null sequence
in $  X$ is norm null.\ So the $ \infty $-Schur property coincides with the Schur property \cite{dm}.
\item A bounded linear operator $ T : X \rightarrow  Y  $ is called  strictly singular, if there is no infinite dimensional subspace $ Z \subseteq X $ such that $ T\vert_{Z} $ is an isomorphism onto its range \cite{AlbKal}.
\item A Banach space $  X$ has the Dunford-Pettis property  (in short $X \in (DPP)$), if for every weakly null sequence $ (x_{n})_n $ in $ X $ and weakly null sequence $ (x^{\ast}_{n})_n $ in $ X^{\ast},$ we have $x^{\ast}_{n}(x_{n})\rightarrow 0$ as $ n\rightarrow\infty$ \cite{di1}.\
\item A Banach space $  X$ has the Dunford-Pettis property of order $p$ (in short $X \in (DPP_{p})$), if for every weakly $ p$-summable sequence $ (x_{n})_n $ in $ X $ and weakly-null sequence $ (x^{\ast}_{n})_n $ in $ X^{\ast},$ we have $x^{\ast}_{n}(x_{n})\rightarrow 0$ as $ n\rightarrow\infty$ \cite{cs}.\

\item A Banach space $ X $ has the $ p $-Gelfand-Phillips property (in short $ p $-$ (GPP) $), if every limited and weakly $ p$-summable sequence in $ X $ is norm null.\  It is  clear that, $ X $ has the $ \infty $-Gelfand-Phillips property if and only if  every limited and weakly null sequence in $ X $ is norm null \cite{fz1}.\
\item 
 
 A bounded linear operator $ T$ from the space of continuous functions defined on a Hausdorff
compact space $  K$ with values in a Banach space $ X $ ( in short $ C(K,X) $), taking values in a Banach space $ Y $ is dominated, if there exists a positive linear functional $ L $ on $  C(K)^{\ast}$ such that $ \Vert T(f) \Vert\leq L(\Vert  f\Vert) , ~~~~ f\in C(K,X)  $ \cite{e1}.
\end{itemize}
\section{ $ p $-Dunford-Pettis relatively compact property}

Here, we introduce the concept of pseudo weakly compact operators of order $ p $ and obtain some characterizations of a Banach space with the  $ p $-$ (DPrcP).  $\ The main goal of this section is to answer to the following question:`` Under which conditions every  dominated operator $ T:  C(K,X)\rightarrow Y$
is $ p $-convergent?".\

\begin{definition}\label{d1}
$ \rm{(i)} $ A bounded linear  operator $ T:X\rightarrow Y $ is called pseudo weakly compact  of order
$  p,$ if $T$ carries  Dunford-Pettis weakly $ p $-compact subset of $ X $ to relatively norm compact in $ Y. $\\
$ \rm{(ii)} $ A sequence $ (x_{n})_{n}$ in a Banach space $ X $ is $ p $-Right null, if $(x_{n})$ is  Dunford-Pettis weakly $ p $-summable.\\
$ \rm{(ii)} $ A sequence $ (x_{n})_{n}$ in a Banach space $ X $ is $ p $-Right Cauchy, if $(x_{n})$ is  Dunford-Pettis weakly $ p $-Cauchy.\
\end{definition}
The class of  pseudo weakly compact  operators of order $ p $ from  $ X $ into $ Y $ is denoted by $ PwC_{p}(X,Y). $\ It is clear that 
if $ 1 \leq p_{1} < p_{2}\leq \infty, $ then $  PwC_{p_{2}}(X,Y)\subseteq  PwC_{p_{1}}(X,Y). $\ In particular, $ PwC(X,Y)\subseteq  PwC_{p}(X,Y). $\ It is clear that the class $  PwC_{p}(X,Y)$ is a closed linear subspace of $ L(X, Y ), $ which has the ideal property, that is, for each $ T\in PwC_{p}(X,Y)$  and each two bounded linear operators $  R$ and $ S, $ which can be composed with
$ T, $ one has that $ R\circ T\circ S $ is also a  pseudo weakly compact  operator.\\
We begin with a simple, but extremely useful, characterization of  pseudo weakly compact operators  of order $ p. $ 
\begin{theorem}\label{t1} Let $ T : X \rightarrow Y $  be a bounded linear operator.\ The following statements  are equivalent:
$ \rm{(i)} $ $  T\in PwC_{p}(X,Y),$\\
$ \rm{(ii)} $ $ T $ maps $ p $-Right null  sequences onto norm null sequences,\\
$ \rm{(iii)} $ $  T$ maps  $ p $-Right Cauchy sequences onto norm convergent sequences.
\end{theorem}
\begin{proof}
(i) $ \Rightarrow $ (ii)
 Let $ (x_{n})_{n} $  be a $ p $-Right null sequence in $ X $ and $ K:=\lbrace x_{n}: n\in \mathbb{N} \rbrace \cup \lbrace 0\rbrace. $\ It is clear that $ K $ is a Dunfotd-Pettis weakly $ p $-compact set in $ X. $\  Since $ T: X\rightarrow Y $ is pseudo weakly compact of order $ p,~ $ $T(K)$ is relatively norm compact in $ Y, $ and so $ ( T(x_{n}))_{n} $ is norm convergent to zero.\\
(ii) $ \Rightarrow $ (iii) Let
$ (x_{n})_{n} $ is a weakly $ p $-Right Cauchy sequence in $ X. $\ Therefore for any two subsequences $  (a_{n})_{n}$ and $ (b_{n})_{n} $ of $ (x_{n})_{n}, $
$(a_{n}-b_{n})_{n}$ is a $ p $-Right null sequence in $ X. $\ So, $ \rm{(ii)} $ implies that
 $ (T(a_{n})-T(b_{n}))_{n} $ is a norm null sequence in $ Y. $\ Hence, $ ( T(x_{n}))_{n} $ is norm convergent.\\
(iii) $ \Rightarrow $ (i)
Suppose that
$ K $ is a Dunford-Pettis weakly $ p $-compact subset of $  X.$\ Let $ (y_{n})_{n} $ be a sequence in $ T(K). $\ Therefore there is a sequence $ (x_{n})_{n}\subseteq K $ such that $ y_{n}=T(x_{n}) ,$ for each $ n\in \mathbb{N} .$\  Since $ K $ is a  weakly $ p $-compact set,
$ (x_{n})_{n} $ has a weakly $ p $-Cauchy subsequence.\ Without
loss of generality we can assume that $ (x_{n})_{n} $ is a $ p $-Right Cauchy sequence.\ Therefore by $ \rm{(iii)} ,$
$ (T(x_{n}))_{n} $ is norm-convergent
in $ Y.$\ Hence $ T(K) $ is relatively  norm compact.\
\end{proof}

Note that Theorem \ref{t1} shows that a bounded liner operator $ T:X\rightarrow Y $ is pseudo weakly compact  of order $  p$ if and only if $ T $ is Dunford-Pettis $  p$-convergent.\ In this note, we use the
terminology  pseudo weakly compact  operators  of order $ p $ instead of  Dunford-Pettis $ p $-converging operators.\
\begin{corollary}\label{c1}
$ \rm{(i)} $ Let   $ T : X \rightarrow Y $ be a bounded linear operator.\ If $ T^{\ast\ast}\in PwC_{p}(X^{\ast\ast},Y^{\ast\ast}) ,$ then 
$ T \in PwC_{p}(X,Y). $\\
$ \rm{(ii)} $ If $ X^{\ast\ast} $ has the $ p $-$ (DPrcP), $ then $ X $ has the same property.\
\end{corollary}

 \begin{theorem}\label{t2} Let   $ T : X \rightarrow Y $ be a bounded linear operator.\ The following statements are equivalent:\\
$ \rm{(i)} $ $  T\in PwC_{p}(X,Y).$\\
$ \rm{(ii)} $ For an arbitrary Banach space $ Z $ and every  weakly limited and weakly $ p $-compact operator  $ S:Z\rightarrow X,$ the
operator $ T\circ S $ is compact.\\
$ \rm{(iii)} $ Same as $ \rm{(ii)} $ with $ Z=\ell_{p^{\ast}} .$
 \end{theorem}
\begin{proof}
(i) $ \Rightarrow $ (ii) Let $ S:Z\rightarrow X $ be a weakly limited and weakly $ p $-compact operator.\ Therefore,  $S(B_{Z})  $ is a Dunford-Pettis weakly $ p $-compact set.\ Hence, $ \rm{(i)} $ implies that $ TS(B_{Z})$  is relatively compact.\\
(ii) $ \Rightarrow $ (iii) It is obvious.\\
(iii) $ \Rightarrow $ (i) Let $ (x_{n})_{n} $ be a $ p $-Right null sequence in $ X. $\ Define the operator $ S : \ell_{p^{\ast}} \rightarrow X$ by
$$ S(\alpha_{1},\alpha_{2},...)=\sum_{n=1}^{\infty} \alpha _{n}x_{n}. $$\
One can see that $ S(B_{\ell_{p^{\ast}}}) =\lbrace   \sum_{n=1}^{\infty} \alpha  _{n}x_{n}:\sum_{n=1}^{\infty}\vert\alpha_{n}\vert \leq1 \rbrace$  is a Dunford-Pettis set in $ X. $\ 
Since $ B_{\ell_{p^{\ast}}} $  is weakly  $ p $-compact \cite{cs}, $S (B_{\ell_{p^{\ast}}})$ is weakly  $ p $-compact.\ Hence, we conclude that  $S(B_{\ell_{p^{\ast}}}) $ is a weakly $ p $-compact set in  $ X. $\ This implies that $ S $ is a weakly limited and weakly $ p $-compact operator.\ 
Hence, $ T\circ S $ is a compact operator and so, $ \lbrace  T(x_{n}):n\in \mathbb{N} \rbrace$ is a compact set.\ \ Furthermore, it is clear that for 
 the canonical basis sequence $ (e^{p^{\ast}}_{n}) $ of $ \ell_{p^{^{\ast}}}, $ the sequence  $(TS(e^{p^{\ast}}_{n}))_{n} =(T(x_{n}))_{n}$ is weakly null.\
This implies that every subsequence of $ \lbrace T(x_{n}):n\in \mathbb{N} \rbrace$ has a subsequence
converging in norm to zero.\ Hence, we have $ \Vert T(x_{n}) \Vert\rightarrow 0.$
\end{proof}

As an immediate consequence of the  Theorems \ref{t1} and \ref{t2}, we can conclude that the following result:
 \begin{proposition}\label{p1}
  Let $ X $ be a Banach space.\ The following assertions  are equivalent:\\
$ \rm{(i)} $   $ X $ has the $ p $-$(DPrcP) ,$\\
$ \rm{(ii)} $  The identity operator  $ id_{X}:X\rightarrow X $ is pseudo weakly compact of order $ p, $\\
 $ \rm{(iii)} $ Every $ p $-Right Cauchy sequence  in $ X $ is norm convergent,\\
$ \rm{(iv)} $ Every   weakly limited and weakly $ p $-compact operator $ S:\ell_{p^{\ast}}\rightarrow X ,$ is compact.
\end{proposition}
\begin{remark}\label{r1}\rm 
$ \rm{(i)} $  If $ X $ has the $ p $-Schur property, then $ X$ has the $ p $-$(DPrcP), $	but
in general the converse is not true.\ For example  for all $ p\geq 2,$ $ \ell_{2} $ has the $ p $-$ (DPrcP), $ while $ \ell_{2}\not\in C_{p}.$\\
 $ \rm{(ii)} $ If $ X $ has the $ (DPrcP), $ then $ X $ has the $ p $-$ (DPrcP). $\ But, the converse is not true,\ For example,
  the space $ L_{1}[0, 1] $ contains no copy of $ c_{0} .$\ Therefore $ L_{1}[0, 1] $ has the $ 1 $-Schur property {\rm (\cite[Corollary 2.9]{dm})}.\ Hence, $ L_{1}[0, 1] $ has the $ 1 $-$ (DPrcP).$\ While, $ L_{1}[0, 1] $ does not have the $ (DPrcP). $\\
$ \rm{(iii)} $ It is clear that $ B_{c_{0}} $ is a Dunford-Pettis set.\ Also, if $ (e_{n})_{n} $ is the standard basis of $ c_{0}, $ then 
 $ (e_{n})_{n}\in \ell_{1}^{w}(c_{0}) $ and so $ (e_{n})_{n}\in \ell_{p}^{w}(c_{0}) $ for all $ p\geq 1. $\ Since  $ \Vert  e_{n} \Vert=1 $  for all $ n\in \mathbb{N} , $ we have $ c_{0}$ does not have the $ p $-$ (DPrcP).$\\
 $ \rm{(iv)} $ There exists a Banach space $ X $ with the $ p $-$ (DPrcP)$ such that if $ Y $ is a closed subspace of it, then the quotient space $ \frac{X}{Y}$ does not have this property.\ For example $c_0$ does not have the $ p $-$ (DPrcP),$ while $ \ell_{1}$ has the $ p $-$ (DPrcP)$  and $c_0$ is isometrically isomorphic to a quotient of $ \ell_{1}$ {\rm (\cite[Corollary 2.3.2]{AlbKal})}.\\
 $ \rm{(v)} $ If $ X $ has the  $ p $-$ (DPrcP), $ then  $ X $ has the $ p $-$ (GPP) ,$  but in general the converse is false.\ For example,  $ c_{0} $ has the $p $-$ (GPP), $ while $c_0$ does not have the $ p $-$ (DPrcP).$\\
 $ \rm{(vi)} $ There exists a Banach space $ X $ with the $ 1 $-$ (DPrcP)  $ such that  $ X^{\ast\ast} $does not have this property.\ For example J.\ Bourgain and F.\ Delbaen \cite{bd} constructed a Banach space $ X_{BD} $ such
that $ X_{BD} $ has the Schur property and $ X^{\ast\ast}_{BD} $ is isomorphically universal for separable
Banach spaces.\ Therefore, there exists a closed subspace $ X_{0} $ 
of $  X^{\ast\ast}_{BD} $  such that  $ X_{0} $ is isomorphic to  $ c_{0}. $\ This implies that  $  X^{\ast\ast}_{BD} $  does not have the  $ 1 $-$ (DPrcP) ,$ while $ X_{BD} $  has the $ 1 $-$ (DPrcP) .$\
\end{remark}

\begin{theorem}\label{t3}
 Let $ X$ be a Banach space.\ The following assertions  are equivalent:\
$ \rm{(i)} $  $ X $ has the $ p $-$ (DPrcP),$\\
$ \rm{(ii)} $ For each Banach space $ Y,~ $ $ PwC_{p}(X, Y ) = L(X, Y ), $\\
$ \rm{(iii)} $ For each Banach space $ Y,~ $ $ PwC_{p}(Y, X ) = L(Y, X ), $\\
$ \rm{(iv)} $ Every closed separable subspace of $ X $ has the $ p $-$ (DPrcP),$\\
$ \rm{(v)} $ $ X $ is the direct sum of two Banach spaces with the $ p $-$ (DPrcP).$
\end{theorem}
\begin{proof}
 (i) $ \Rightarrow $ (ii). Assume that $ T \in L(X, Y )  $ and that $ (x_{n})_{n} $ is a  $ p $-Right null  sequence in $ X. $\  Since $ X $ has the $ p $-$ (DPrcP) ,$ $ (x_{n})_{n} $ is norm null.\ So $ \Vert T(x_{n}) \Vert\rightarrow 0;$ that is $  PwC_{p}(X, Y ) = L(X, Y ).$\\
(ii) $ \Rightarrow $ (i). If $ Y = X, $ then (ii) implies that the identity operator on $ X $ is  pseudo weakly compact of order $  p.$\ Hence,
$ X $ has the $ p $-$ (DPrcP). $\\
The proof of the equivalence of $ \rm{(i)} $ and $ \rm{(iii)} $ is similar.\\
$\rm{(i)}\Rightarrow \rm{(iv)}  $ If $ X $ has the $ p $-$ (DPrcP), $ then any closed subspace $ Z $ of $  X$ has the $ p $-$ (DPrcP), $ since any  $ p $-Right null sequence in
$ Z $ is also a $ p $-Right null sequence in $ X. $\\
$\rm{(iv)}\Rightarrow \rm{(i)}  $ Suppose that any closed separable subspace of $  X$ has the $ p $-$ (DPrcP) $  and let $ (x_{n})_{n} $
be a $ p $-Right null sequence in $ X $ which is not norm null. Let $  Y=[x_{n}]$  be the
closed linear span of $ (x_{n})_{n} .$ Note that $  Y$ is a separable subspace of $ X $. By (\cite[Theorem 1.6]{hm}) there is a
separable subspace $ Z $ of $  X$ and an isometric embedding $ j : Z^{\ast} \rightarrow X^{\ast} $ that define by $ J(z^{\ast})(z)=z^{\ast}(z) $ for each $ z \in Z $ and $  z ^{\ast}\in Z^{\ast}.$\ By our hypothesis, $  Z$ has the  $ p $-$ (DPrcP). $\ Therefore  $ \lbrace x_{n} :n\in\mathbb{N}\rbrace $ is not a Dunford-Pettis in  $ Z. $\ Hence there is a weakly null sequence $ (z_{n}^{\ast}) $ in $ Z^{\ast} $ and a subsequence $ (x_{k_{n}})_{n} $ of $ (x_{n})_{n}, $ which we still
denote it by $ (x_{n})_{n}, $ such that $ z_{n}^{\ast}(x_{n})=1 $
 for each $ n\in\mathbb{N}. $\
Let $x_{n}^{\ast}=j(z_{n}^{\ast})  $ for each $ n\in\mathbb{N}. $\ It is clear that  $ (x_{n}^{\ast})_{n} $
 is weakly null in $ X^{\ast} $ and for each $  n,$ $ x_{n}^{\ast}(x_{n}) =z_{n}^{\ast}(x_{n})=1.$\ Hence {\rm (\cite[Theorem 1]{An})} yields that
$  \lbrace x_{n}: n\in\mathbb{N}\rbrace $ is not a Dunford-Pettis set in  $ X, $ 
which is a contradiction.\\
$\rm{(i)}\Rightarrow \rm{(v)} $ Note that $ X=X\bigoplus \lbrace 0\rbrace. $\\
$\rm{(v)}\Rightarrow \rm{(i)} $ Let $ X = Y \bigoplus Z $ such that $ Y $ and $ Z $ have the $ p $-$  (DPrcP).$\ Consider
the projections $ P_{1} : X \rightarrow Y  $ and $ P_{2} : X \rightarrow Z. $\  Suppose that $ K $ is a  Dunford-Pettis weakly $ p $-compact subset of $ X. $\ Clearly,  $ P_{1} (K)$ is a Dunford-Pettis weakly $ p $-compact subset of $ Y,$ and so it
is a norm compact set in  $ Y. $\ Similarly $ P_{2} (K)$ is a norm compact set in $ Z. $\ Also it is clear that any sequence $ (x_{n})_{n} \subseteq K $ can
be written as $ x_{n} = y_{n} + z_{n}, $ where $ y_{n}\in P_{1} (K) $ and  $ z_{n}\in P_{2} (K). $\ Thus, there are
subsequences $ (y_{n_{k}} )_{k} $ and $ (z_{n_{k}} )_{k} $ and  $ y \in P_{1} (K) $ and $ z\in P_{2} (K)$ such that  $ x_{n_{k}}=y_{n_{k}} +z_{n_{k}}\rightarrow y+z.$\ Since, $ K $ is a weakly $ p $-compact set, $ y+z\in K $
 and so, $  K$ is a  norm compact set.\
\end{proof}
\begin{theorem}\label{t4} If $ X $ has the $ p $-$ (DPrcP) ,$ then the following statements hold:
$ \rm{(i)} $ $ \displaystyle \lim_{n\rightarrow\infty} x_{n}^{\ast}(x_{n})=0,$  for every  $  p$-Right Cauchy sequence $ (x_{n})_{n} $ in $  X$ and every weakly null sequence $ (x^{\ast}_{n})_{n} $ in $  X^{\ast},$\\
$ \rm{(ii)} $ $ \displaystyle \lim_{n\rightarrow\infty} x_{n}^{\ast}(x_{n})=0,$ for every  $  p$-Right null sequence $ (x_{n})_{n} $ in $  X$ and every weakly null sequence $ (x^{\ast}_{n})_{n} $ in $  X^{\ast},$\\
$ \rm{(iii)} $ $ \displaystyle \lim_{n\rightarrow\infty} x_{n}^{\ast}(x_{n})=0,$ for every  $  p$-Right null sequence $ (x_{n})_{n} $ in $  X$ and every weakly Cauchy sequence $ (x^{\ast}_{n})_{n} $ in $  X^{\ast}.$\
\end{theorem}
\begin{proof}  
(i) Let $ (x_{n})_{n} $ be a  $  p$-Right Cauchy sequence $ (x_{n})_{n} $ in $  X$ and $ (x^{\ast}_{n})_{n} $ be a weakly null sequence  in $  X^{\ast}.$\ Define a bounded linear operator $ T: X\rightarrow c_{0} $ by $ T(x)=(x^{\ast}_{n}(x))_{n} .$\ By Theorem \ref{t3}, $ T\in PwC_{p}(X,c_{0}). $\ Therefore, Theorem \ref{t1} implies that $ (T(x_{n}))_{n} $ converges to some $ \alpha=(\alpha_{n})_{n}\in c_{0} $ in norm.\ For every $ \varepsilon> 0 $ there exists a positive integer $ N_{1} $ such that $ \Vert  T(x_{n})-\alpha \Vert < \frac{\varepsilon}{2} $  for all $ n > N_{1}. $\  Since $ \alpha\in c_{0}, $ we choose another positive integer $ N_{2} $ such that $ \vert \alpha_{k} \vert < \frac{\varepsilon}{2} $  for all $ k> N_{2} .$\ Hence, we have $ \vert    x_{n}^{\ast}(x_{n}) \vert <\varepsilon $ for all $ n > max\lbrace N_{1},N_{2}\rbrace. $\ Thus  $  \displaystyle\lim_{n\rightarrow\infty} x_{n}^{\ast}(x_{n})=0.$\\
(ii) It is trivial.\\
(iii)  Suppose there exists a  $ p $-Right null sequence $ (x_{n})_{n} $ in $ X $ and there exists a weakly Cauchy sequence $  (x_{n}^{\ast})_{n}$ in $ X^{\ast} $ such that $ \vert x_{n}^{\ast}(x_{n}) \vert>\varepsilon,$ for some $ \varepsilon>0 $ and all $ n\in\mathbb{N}.$\  Since $ (x_{n})_{n} $ is weakly  $ p $-summable and in particular weakly null, there exists a subsequence $ (x_{k_{n}})_{n} $ of $ (x_{n})_{n} $ such that $ \vert x_{n}^{\ast}(x_{k_{n}}) \vert <\frac{\varepsilon}{2} $ for all $ n\in \mathbb{N}. $\ Since $ (x^{\ast}_{n})_{n} $ is weakly Cauchy, we see that  $ (x^{\ast}_{k_{n}}-x^{\ast}_{n})_{n} $ is weakly null.\  Now, by $ \rm{(ii)}, $ we have $ \lim_{n\rightarrow\infty} (x^{\ast}_{k_{n}}-x^{\ast}_{n})(x_{k_{n}})=0.$\ This implies that  $ \vert  (x^{\ast}_{k_{n}}-x^{\ast}_{n})(x_{k_{n}})   \vert <\frac{\varepsilon}{3} $ for $ n $ large enough.\ But for such $ n$'s, we have $$ \varepsilon <\vert x^{\ast}_{k_{n}}(x_{k_{n}})  \vert \leq \vert ( x^{\ast}_{k_{n}} -x^{\ast}_{n})(x_{k_{n}})  \vert+ \vert x^{\ast}_{n}(x_{k_{n}}) \vert  <\frac{5\varepsilon}{6},$$  which is a contradiction.\
\end{proof}
Recall that \cite{MZ}, if $ \mathcal{U} $ is an arbitrary Banach operator ideal and
$ \mathcal{M} $ is a closed subspace
of $\mathcal{U}(X,Y) ,$ then for arbitrary elements $ x\in X $ and  $ y^{\ast}\in Y^{\ast}, $ the evaluation operators $ \varphi_{x} : \mathcal{M}\rightarrow Y$ and
$\psi_{y^{\ast}} : \mathcal{M} \rightarrow X^{\ast} $ on $
\mathcal{M} $  are defined by $ \varphi_{x}(T) = T(x)$ and $
\psi_{y^{\ast}}(T) = T^{\ast}(y^{\ast}) $ for $ T \in \mathcal{M}.$\ 
The following result shows that the pseudo weakly compact  of order $ p$ of all evaluation
operators of a closed subspace $\mathcal{M} \subseteq\mathcal{U}(X,Y) ,$ is a necessary condition for
the $ p $-$ (DPrcP)$ of $\mathcal{M}. $
\begin{corollary}\label{c2}
  If  $\mathcal{M} $ is a closed subspace of operator ideal $ \mathcal{U}(X,Y)$ that has the
 $ p $-$ (DPrcP),$ then all of the evaluation operators $ \varphi_{x} :\mathcal{M} \rightarrow Y$ and $\psi_{y^{\ast}} : \mathcal{M} \rightarrow X^{\ast} $ are pseudo weakly compact  of order $ p. $
\end{corollary}
\begin{theorem}\label{t5}
Let $ X $ and $  Y$ be two Banach spaces such that $ Y $ has the Schur property.\ If
$ \mathcal{M} $ is a closed subspace of $  \mathcal{U}(X,Y)$ such that each evaluation operator $  \psi_{y^{\ast}}$ is pseudo weakly compact of order $ p $ on
$ \mathcal{M}, $ then $ \mathcal{M} $  has the $ p $-$  (DPrcP).$
\end{theorem}
\begin{proof}
 Suppose that $ \mathcal{M} $ does not have the $ p $-$  (DPrcP).$\ Then there is a  $ p $-Right null sequence
$ (T_{n})_{n} $ in $ \mathcal{M} $ such that $ \Vert T_{n} \Vert \geq \varepsilon $ for all positive integer $  n$ and some $ \varepsilon >0. $\ We can choose a sequence $ (x_{n})_{n} $
in $ B_{X} $ such that $ \Vert T_{n}(x_{n}) \Vert \geq \varepsilon.$\ In addition, for each $ y^{\ast}\in Y^{\ast}, $ the evaluation operator is $  \psi_{y^{\ast}}$ is pseudo weakly compact of order $ p .$\
Therefore  $ \Vert T_{n}^{\ast}(y^{\ast})  \Vert\rightarrow 0. $\
So, $ \vert \langle  y^{\ast} , T_{n}(x_{n}) \rangle\vert \leq \Vert T^{\ast}_{n}(y^{\ast})   \Vert \Vert x_{n} \Vert\rightarrow 0. $\
 Hence $ (T_{n}(x_{n}))_{n} $ is weakly null in $ Y, $
and so is norm null, which is a contradiction.\
\end{proof}
 As an immediate consequence of the  part $ \rm{(ii)} $ in $ \rm{Exercise } {~4}$ of Chapter $ \rm{VII }$ \cite{di1}, we can conclude the following result.\ The proof is simple and left to the reader.
\begin{lemma}\rm \label{l1}
Let $ (\alpha_{n})_{n} \in\ell_{\infty}.$\ The operator $ T:c_{0}\rightarrow c_{0}$ defined by
$T(x_{1},x_{2},...)=(\alpha_{1}x_{1},\alpha_{2}x_{2},...) $ is  compact  if and only if $ (\alpha_{n})_{n} \in c_{0}. $
\end{lemma}
A bilinear operator
$ \phi:X\times Y\rightarrow Z $
is called separately compact if for each fixed $ y \in Y , $ the linear operator $ T_{y}: X\rightarrow Z : x\mapsto \phi (x,y)$  and for each fixed $ x \in X , $ the linear operator $ T_{x}: Y\rightarrow Z : y\mapsto \phi (x,y) $
  are compact.\
\begin{proposition}\label{p2} 
If  every symmetric bilinear separately compact operator $ S:X\times X\rightarrow c_{0} $ is pseudo weakly compact of order $ p ,$ then $ X $ has the  $ p $-$ (DPrcP) .$\
\end{proposition}
\begin{proof}
 If $ X$ does not have the $ p $-$ (DPrcP),$ then there is a  $ p $-Right null sequence $ (x_{n})_{n} $ in the unit ball of $ X. $\ Using
the Bessaga-Pelczy$\acute{n}$ski selection principle \cite{AlbKal}, we can pick a basic subsequence
of it, which we will call $  (x_{n})_{n}.$\ Let $  Z$ be the closed subspace
of $ X $ generated by $ (x_{n})_{n} .$\ Let us define the operator $  T: Z\rightarrow  c_{0}$ by
$ T(x_{n})=e_{n} $ where $ (e_{n})_{n} $ is the canonical basis of $ c_{0}. $\ We can now use
Sobczyk’s theorem \cite{AlbKal} to extend $  T$ to the whole of $ X. $\ For convenience, we denote its  extension again by $ T. $\ Now, consider the symmetric bilinear operator
$ S: X\times X\rightarrow c_{0} $ given by $ S(x, y) = T(x) .T(y),  $ the product is the component wise product in $ c_{0}. $\
We show that for each fixd  $ y \in X, $ the operator $ S(., y) : X \rightarrow c_{0}  $ given
by  $ x $ maps to $ S(x, y) $ is compact.\ Since it can be decomposed as $ \delta_{x}\circ T, $ where $   \delta_{x}: c_{0}\rightarrow c_{0}$ denotes the
diagonal operator given by $ z\mapsto T(x).z. $\  It is clear  that $  \delta_{x} $ is compact (see Lemma \ref{l1}) and so,  $ S(., y)=\delta_{x}\circ T $ is compact.\ Proceeding analogously with
the other variable, we infer that $ S $ is separately compact.\ On the
other hand, $  S$ is not a pseudo weakly compact operator of order $ p ,$ since it maps the sequence
$ (x_{n}, x_{n})_{n}, $ which is  $ p $-Right null in $ X\times  X $ into the basis of $ c_{0}, $ which is a contradiction.\
\end{proof}
\begin{proposition}\label{p3} 
 Suppose that $ T \in PwC_{p}(X, Y ) $ is not strictly singular.\ Then, $ X $
and $ Y $ contain simultaneously some infinite dimensional closed subspaces with
the $ p $-$ (DPrcP). $\
\end{proposition}
\begin{proof}
 Suppose that $ T $ has a bounded inverse on the closed infinite dimensional subspace
$  Z$ of $  X.$\ If $ (x_{n})_{n} $ is a  $ p $-Right null sequence in $ Z, $
 then $ (x_{n})_{n} $ is a $ p $-Right null sequence in $ X. $\ By assumption, $ \Vert T(x_{n})\Vert\rightarrow 0 $ and so $ \Vert x_{n} \Vert\rightarrow 0.$\
 Hence, $  Z$ has the $ p $-$ (DPrcP). $\ Similarly, we can see that  $ T(Z)  $ has the same property.
\end{proof}
It is clear that every $p$-convergent operator is pseudo weakly compact of order $ p ,$ but in general the converse is not true.\ For example, the identity operator $ id_{\ell_{2}}:\ell_{2}\rightarrow\ell_{2} $ is weakly compact and so is pseudo weakly compact  of order $ 2, $ while it is not $ 2$-convergent.\\
 Here, we give a characterization
of those Banach spaces in which the converse of the above assertion
holds.
\begin{theorem}\label{t6} If $ X $ is a Banach space, then $ X $ has the  $ (DPP_{p}) $ if and only if for each
Banach space $ Y, $  $ C_{p}(X,Y)=PwC_{p}(X,Y) .$
\end{theorem}
\begin{proof}
  Suppose that $ X\in(DPP_{p}) $ and $ T\in PwC_{p}(X, Y ).$\ If $ (x_{n})_{n} $ is a weakly $ p $-summable sequence in $ X, $ then  $ \lim_{n\rightarrow\infty} x^{\ast}_{n}(x_{n})=0$ for each weakly null $ (x^{\ast}_{n})_{n} $ in $ X^{\ast}. $\ Hence, {\rm (\cite[Theorem 1]{An})} implies that $ \lbrace x_{n}:n\in\mathbb{N}\rbrace$ is a Dunford-Pettis set.\ Hence, $ (x_{n})_{n} $ is a $ p $-Right null and so, the sequence $ (T(x_{n}))_{n} $ is norm null.\ Therefore, $ T $ is $ p $-convergent.\\
Conversely,
Suppose that $ X $ does not have the  $ (DPP_{p}) .$\ Therefore by {\rm (\cite[Theorem 3.1]{ccl})} there exists a  weakly $ p $-summable sequence $ (x_{n})_{n} $ in $ X $ and there exists a weakly null sequence $  (x_{n}^{\ast})_{n}$ in $ X^{\ast} $ such that $ \vert x_{n}^{\ast}(x_{n}) \vert>\varepsilon,$ for some $ \varepsilon>0 $ and all $ n\in\mathbb{N}.$\ Define  the operator $ T:X\rightarrow c_{0} ,$ as  $ T(x)=(x_{n}^{\ast}(x))_{n} .$\ It is clear that $ T $ is weakly compact,  and so $ T$ is pseudo weakly compact  of order $ p.$\ Since  $ C_{p}(X,Y)=PwC_{p}(X,Y) ,$ $ T $  is $ p $-convergent.\ But,  $ (x_{n})_{n} $ is a weakly $ p $-summable sequence in $ X $ and $ \Vert T(x_{n})\Vert\geq \vert x^{\ast}_{n}(x_{n})\vert>\varepsilon, $ for all $ n\in \mathbb{N},$\ which is a contradiction.\
\end{proof}
\begin{corollary}\label{c3}
$ \rm{(i)} $ If $ X $ is an arbitrary Banach space, then   $ C_{1}(X,Y)=PwC_{1}(X,Y),$  for each Banach space  $ Y. $\\
$ \rm{(ii)} $ A Banach space $ X $  has  both the  $ p $-$ (DPrcP) $ and  $ (DPP_{p}) $ if and only if $ X $ has the  $ p $-Schur property.\
\end{corollary}
 Let  $ \mathcal{M} $ be a bounded subspace of $ \mathcal{U}(X,Y). $\ The point evaluation sets related to $ x \in X $ and $ y^{\ast} \in Y^{\ast} $ are the images of the closed
unit ball $ B_{\mathcal{M}} $ of $ \mathcal{M}, $ under the evaluation operators $ \phi_{x} $ and $\psi_{y^{\ast}}  $ are denoted by $\mathcal{M}_{1}(x)$ and $
\widetilde{\mathcal{M}_1}(y^{\ast})$  respectively \cite{MZ}.\\

 By a similar technique as in {\rm (\cite[Theorem 2.2]{w})} with minor modification, we obtain the following result.
\begin{theorem}\label{t7}  Suppose that $ X^{\ast\ast}$ and $ Y^{\ast}  $ have the $  (DPP_{p}). $\ If $  \mathcal{M}$ is a  closed subspace  of $ \mathcal{U}(X,Y) $ such that $ \mathcal{M}^{\ast} $ has the $ p $-$(DPrcP), $ then of all the point evaluations $ \mathcal{M}_{1}(x) $ and $\widetilde{\mathcal{M}} _{1}(y^{\ast})$ are $ p $-$ (V^{\ast}) $ sets in $ Y $ and $ X^{\ast} $ respectively, where $ x \in X $ and $ y^{\ast} \in Y^{\ast}. $
\end{theorem}
\begin{proof} Since $ \mathcal{M}^{\ast} $ has the  $ p $-$( DPrcP) ,$
 Theorem
\ref{t3}, implies that the adjoint operators $ \varphi^{\ast}_{x}
:Y^{\ast}\rightarrow \mathcal{M}^{\ast} $ and $
\psi^{\ast}_{y^{\ast}}:X^{\ast\ast}\rightarrow
\mathcal{M}^{\ast}$  are  pseudo weakly compact
of order $ p.$ On the other hand, $ X^{\ast\ast}$ and $ Y^{\ast} $ have the $  (DPP_{p}). $\ Hence, Theorem \ref{t6} implies that $  \varphi^{\ast}_{x} $ and $\psi^{\ast}_{y^{\ast}}  $
are $ p $-convergent.\ Suppose
that
$(y^{\ast}_{n})_{n} $ is a weakly $ p $-summable sequence in $ Y^{\ast}. $\ Therefore we have:\
 $$\lim_{n\rightarrow\infty}\sup\lbrace \vert y^{\ast}_{n}(T(x))\vert: T\in B_{\mathcal{M}} \rbrace =\lim_{n\rightarrow\infty} \sup \lbrace \vert \varphi^{\ast}_{x}(y^{\ast}_{n})(T)\vert:T\in B_{\mathcal{M}}\rbrace=\lim_{n\rightarrow\infty} \Vert \varphi^{\ast}_{x}(y^{\ast}_{n})\Vert=0 ,$$    for all $ x \in X. $\ Hence  $ (y^{\ast}_{n})_{n} $
 converges uniformly on $ \mathcal{M}_{1} (x).$\ This shows that $ \mathcal{M}_{1} (x)$ is a $ p $-$ (V^{\ast}) $ set in $Y,$  for all $ x\in X.$\ A similar proof shows that $ {\widetilde{\mathcal{M}}}_{1}(y^{\ast}) $ is a $ p $-$ (V^{\ast}) $ set in $ X^{\ast}, $ for all $y^{\ast} \in Y^{\ast}.$
 \end{proof}
Let $ (X_{n})_{n\in \mathbb{N}}  $ be a sequence of Banach spaces.\ If $ 1\leq r <\infty $ the space of all vector-valued sequences $ (\displaystyle\sum_{n=1}^{\infty}\oplus X_{n})_{\ell_{r}} $ is called, the infinite direct sum of $ X_{n} $ in the sense of $ \ell_{r}, $ consisting of all sequences $ x=(x_{n})_{n} $ with values in $X_{n} $ such that $ \Vert x \Vert_{r}=(\displaystyle\sum_{n=1}^{\infty}\Vert x_{n}\Vert^{r} )^{\frac{1}{r}}<\infty.$\
\begin{proposition}\label{p4}
Let  $ (X_{n})_{n\in \mathbb{N}} $ be a family of Banach spaces.\ Then $ X_{n}$ has the $ p $-$ (DPrcP) $ for all $ n\in\mathbb{N} $ if and only if  $ (\displaystyle\sum_{n=1}^{\infty}\oplus X_{n})_{\ell_{1}} $ has the same property.
\end{proposition}
\begin{proof}
It is clear that if $  X=(\displaystyle\sum_{n=1}^{\infty}\oplus X_{n})_{\ell_{1}}$ has the $ p $-$ (DPrcP) ,$ then  every closed subspace of $ X $ has the $ p $-$ (DPrcP) . $\ Hence $ X_{n}$ has the $ p $-$ (DPrcP) $ for all $ n\in\mathbb{N} .$\ Now,  suppose that  $ (x_{n})_{n} $ is a $ p $-Right null sequence in $ X ,$
 where $ x_{n}=(b_{n,k})_{k\in \mathbb{N}} .$\ It is clear  that $  (b_{n,k})_{k\in \mathbb{N}}$ is a $ p $-Right null sequence in $ X_{k} $
 for all $ k\in \mathbb{N}. $\ Since $ X_{k} $ has the $ p $-$ (DPrcP) ,$  $\Vert  b_{n,k} \Vert_{X_{k}}\rightarrow 0  $ as $ n\rightarrow\infty $ for all $ k\in \mathbb{N} .$\
Using the techniques which used in {\rm (\cite[Lemma, page 31]{d})},
we can conclude that the sum $\Vert x_{n}\Vert_{1}=\displaystyle\sum_{k=1}^{\infty}\Vert b_{n,k} \Vert_{X_{k}}$ is converges uniformly in $ n. $\ Hence,  $ \displaystyle\lim_{n\rightarrow\infty}\Vert x_{n} \Vert_{1}=\displaystyle\lim_{n\rightarrow\infty}\sum_{k=1}^{\infty}\Vert b_{n,k} \Vert_{X_{k}}=\displaystyle\sum_{k=1}^{\infty}\lim_{n\rightarrow\infty}\Vert b_{n,k} \Vert_{X_{k}}=0 .$\
\end{proof}
\begin{remark}\label{r2}\rm
It is not necessary that $(\displaystyle\sum_{n=1}^{\infty}\oplus X_{n})_{\ell_{\infty}} $ has the $ p $-$(DPrcP). $\ For example, let $  X_{k}=\mathbb{R}.$\ It is clear that $ X_{k}$ has the $ p $-$ (DPrcP) $ for each $  k,$ but $ (\displaystyle\sum_{n=1}^{\infty}\oplus X_{n})_{\ell_{\infty}}\cong \ell_{\infty}$ does  not have the $ p $-$ (DPrcP). $
\end{remark}
 If $ X \in (DPP)  $ and $ Y \in (DPrcP), $ then  the dominated operators from 
$ C(K, X) $ spaces taking values in $ Y $ with the $ (DPrcP)$  are completely continuous (see {\rm (\cite[Theorem 11]{e1})}).\
Here, by a similar technique we state that dominated operators from $ C(K, X) $ spaces taking values in Banach space with the $ p $-$ (DPrcP) $ are  $ p $-convergent.\\
\begin{theorem}\label{t8}
Suppose  that $ Y$ has the $ p $-$ (DPrcP)$ and  $ K $ is a compact Hausdorff space.\ If $ X $ has the $ (DPP_{p}), $ then any dominated operator $ T $ from $ C (K, X) $ into $ Y $ is  $ p $-convergent.
\end{theorem}
\begin{proof}
Let $ T:C(K,X) \rightarrow Y$ be an arbitrary dominated operator.\ By
Theorem 5 in Chapter $ \rm{III} $ of \cite{di}, there is a function $ G $ from $ K $ into $ L(X, Y^{\ast\ast}) $ such that\\
$\rm{( i)} $ $ \Vert G(t)\Vert =1 ~\mu. a.e.  $ in $ K. $ i.e.; $ \mu(\lbrace t\in K : \Vert G(t)\Vert \not =1\rbrace)=0. $\\
$\rm{( ii)} $ For each $ y^{\ast}\in Y^{\ast} $ and $ f \in C (K, X), $ the
function $ y^{\ast}( G(.)f(.)) $ is $ \mu
$-integrable and moreover
\begin{center}
$   y^{\ast}(T(f)) =\int_{K}
y^{\ast}( G(t)f(t)) d\mu ~~ $ for $ ~~f\in C(K,X). $
\end{center}
Where $ \mu $ is the least regular Borel measure dominating $ T.$\ Consider a  weakly  $ p $-summable sequence  $ (f_{n})_{n} $ in $ C(K, X) .$\ Since
continuous linear images of weakly $ p$-summable sequences are weakly $ p$-summable sequences, $ (T(f_{n}))_{n} $ is a weakly $ p$-summable sequence in $ Y. $\ Now, we show that $ \lbrace T(f_{n})) :n\in\mathbb{N}\rbrace $ is a Dunford-Pettis set in $ Y. $\ For this purpose, we consider a weak null sequence
 $ (y_{n}^{\ast})_{n}$ in $Y^{\ast}. $\ It is not difficult to show that,
 for each $t \in K,~  (G^{\ast}(t) y^{\ast}_{n})_{n} $
 is weakly null in $ X^{\ast} $ and
  $ (f_{n}(t))_{n} $ is a weakly $ p $-summable sequence in $ X. $\
 Since $ X \in(DPP_{p}),$ we have :
\begin{center}
$  y_{n}^{\ast}( (G(t) f_{n}(t)) )= G^{\ast}(t) y_{n}^{\ast}( f_{n}(t)) \rightarrow 0. $
\end{center}
 Moreover, there exists a constant $ M>0 $ such that
 $ \vert   y_{n}^{\ast}(G(t)f_{n}(t))\vert\leq M  $ for all $ t\in K $ and $ n\in \mathbb{N}. $\ The Lebesgue dominated convergent Theorem, implies that:
 \begin{center}
 $ \displaystyle\lim_{n\rightarrow\infty} y_{n}^{\ast}(T(f_{n})) =\lim_{n\rightarrow\infty}\int_{K} y_{n}^{\ast}( G(t)f_{n}(t) )d\mu =0. $
 \end{center}
Hence,  $ \lbrace T(f_{n}) :n\in\mathbb{N}\rbrace
$ is a Dunford-Pettis set in $  Y$\ {\rm (\cite[Theorem 1]{An})}.\ It is clear that  $ \Vert T(f_{n}) \Vert\rightarrow 0.$\ Since $ Y $ has the $ p $-$ (DPrcP). $\
 \end{proof}
 \begin{proposition}\label{p5}
If $ X^{\ast}$ has the $ p $-$(DPrcP) $ and $ Y $ has the Schur property, then $ L(X,Y)=K(X,Y). $
\end{proposition}
\begin{proof}
 Suppose that there exists a bounded linear  operator $ T :X\rightarrow Y $ which is not compact.\
Since  $  Y$ has the Schur property,  there is a bounded sequence $ (x_{n})_{n} $ in $ X $ that has no weakly Cauchy subsequence (see  Corollary 4 of \cite{l}).\ Thus, Rosenthal's $\ell_{1}$-theorem implies that $  X$ contains a
copy of $ \ell_{1} .$\ Hence  $ X^{\ast} $ contains a copy of $ c_{0} ,$ which is a contradiction,
since $ X^{\ast} $ has the $ p $-$ (DPrcP). $
\end{proof}
The authors \cite{e1,g18}, studied the lifting of the $ (DPrcP) $ from
$ X^{\ast} $ and from a Banach space $  Y$ to the space
$ K(X,Y) .$\ 
Here, we obtain some suitable conditions on $  X$ and
$  Y$ such that $ L(X, Y ), $ and some its subspaces have the $ p $-Dunford-Pettis relatively compact property.
\begin{theorem}\label{t9}
Let $ X $ and $ Y $ be two Banach spaces such that $ Y $ has the Schur property.\ If
$ \mathcal{M} $ is a closed subspace of $ L(X, Y ) $ such that each evaluation operators $ \psi_{y^{\ast}} $ is a pseudo weakly compact operator of order $ p $ on $ \mathcal{M}, $ then
$ \mathcal{M}$ has the $ p $-$ (DPrcP). $
\end{theorem}
\begin{proof}
If $ \mathcal{M}$ does not have the $ p $-$ (DPrcP), $ then  there is a Dunford-Pettis  weakly $ p $-summable sequence
 $ (T_{n})_{n} $ in $ \mathcal{M}$ that is not norm null and by passing to a subsequence, we may assume that
$ \Vert T_{n} \Vert >\varepsilon$ for all integer $  n$ and some $ \varepsilon >0. $\ Therefore, there exists a sequence $ (x_{n})_{n} $ in $ B_{X} $ such that $ \Vert T_{n}x_{n}\Vert >\varepsilon, $
for all $ n$ and some $ \varepsilon >0. $\ In addition, for each $ y^{\ast}\in Y^{\ast},$ the evaluation operator $  \psi_{y^{\ast}} :\mathcal{M}\rightarrow X^{\ast} $
 is pseudo weakly compact operator of order $ p ,$ so $ \Vert T_{n}^{\ast}y^{\ast}\Vert =\Vert\psi_{y^{\ast}}(T_{n})\Vert\rightarrow 0$
and then
$$\vert \langle T_{n}(x_{n}), y^{\ast} \rangle\vert\leq \Vert T_{n}^{\ast}(y^{\ast}) \Vert\rightarrow 0. $$
Hence, the sequence $ (T_{n}(x_{n})) _{n}$ is weakly null and so norm null, which is a  contradiction.\ Therefore,    $ \mathcal{M}$ has the $ p $-$ (DPrcP). $
\end{proof}
Recall that \cite{g6}, the class of $ w^{\ast} $-$ w$ continuous (resp., compact) operators  from $ X^{\ast} $  to  $ Y $ will
be denoted by $L_{w^{\ast}}(X^{\ast},Y)  $ (resp., $K_{w^{\ast}}(X^{\ast},Y)  $).\
By a similar method, we obtain a sufficient condition for the $ p $-$ (DPrcP) $ of closed subspaces of $ L_{w^{\ast}}(X^{\ast},Y).$\
Since the proof of the following result is similar to the proof of Theorem \ref{t9}, we omit its proof.
 \begin{theorem}\label{t10}
Let $ X $ and $ Y $ be two Banach spaces such that $ X$  has the Schur property.\
If $ \mathcal{M} $ is a closed subspace of $ L_{w^{\ast}}(X^{\ast},Y)$ such that each evaluation operators  is  $\psi_{y^{\ast}} $ is a pseudo weakly compact operator of order $ p $ on $ \mathcal{M}, $ 
then  $ \mathcal{M} $ has the $ p $-$ (DPrcP) .$
 \end{theorem}

\begin{corollary}\label{c4} Suppose that $ X $and $ Y $  are  Banach spaces.\ The following statements  hold:\\
 $ \rm{(i)} $ If $ X^{\ast} $ has the $ p $-$(DPrcP) $ and $ Y  $ has the Schur property, then $  L(X, Y ) $ has the $ p $-$ (DPrcP).$ \\
$ \rm{(ii)} $ If $ X $ has the $  p $-$ (DPrcP) $ and $ Y  $ has the Schur property, then $ L_{w^{\ast}}(X^{\ast}, Y ) $ has the   $ p $-$ (DPrcP). $\\
 $ \rm{(iii)} $  If $ X^{\ast} $ has the $ p $-$ (DPrcP), $ then $\ell_{1}^{w}  (X^{\ast})$
 has the same property.
\end{corollary}
\section{ $ p $-Sequentially Right  property  on Banach spaces}

In this section, we study  the notion of $ p $-sequentially Right property and characterize this property in terms of weakly compact operators.\ Also,
we investigate the stability of $  p$-sequentially Right property for some subspaces of bounded linear
operators, projective tensor product and injective tensor product between Banach spaces $ X $ and $ Y. $\ Finally,
 by introducing the notion of weak p-sequentially Right property, we obtain a characterization of Banach spaces which have this property.\\ 
 
The following Proposition gives some additional properties of $ p $-Right sets
in a topological dual Banach space.
\begin{proposition}\label{p6}
$ \rm{(i)} $  Absolutely closed convex hull of a $ p $-Right subset of a dual Banach space is a $ p $-Right set,\\
$ \rm{(ii)} $ Every weak$ ^{\ast} $ null sequence in dual Banach space is a $ p $-Right set,\\
$ \rm{(iii)} $ A bounded
subset $ K $ of $ X^{\ast}$ is a $ p $-Right set if and only if for each sequence $ (x^{\ast}_{n})_{n} $ in $ K, $ $ x^{\ast}_{n}(x_{n})\rightarrow 0, $ for every $  p $-Right null
sequence $ (x_{n})_{_{n}} $ of $  X,$\\
$ \rm{(iv)} $
If $ (x^{\ast}_{n})_{n} $ is a norm bounded
sequence of $ X^{\ast} ,$ then the subset  $ \lbrace x^{\ast}_{n}: n \in \mathbb{N}\rbrace $ is a $ p $-Right set if and only if
  $ x^{\ast}_{n}(x_{n})\rightarrow 0, $ for every $  p $-Right null
sequence $ (x_{n})_{_{n}} $ of $  X.$\\
\end{proposition}

\begin{theorem}\label{t11}
A  bounded linear  operator $ T :X \rightarrow Y  $ is pseudo weakly compact  of order $ p$ if and only if $ T^{\ast} $ maps bounded subsets of  $ Y^{\ast} $ onto $  p$-Right subsets of $ X^{\ast}. $
\end{theorem}
\begin{proof}
Let $ (x_{n})_{n} $ be a  $ p $-Right null sequence in $ X. $\ Then the  equalities:\
$$\Vert T(x_{n})\Vert=\displaystyle\sup_{y^{\ast}\in B_{Y^{\ast}}}\vert T^{\ast}(y^{\ast})(x_{n})\vert=\displaystyle\sup_{y^{\ast}\in B_{Y^{\ast}}}\vert y^{\ast}(T(x_{n}))\vert$$
 deduces the proof.
\end{proof}
\begin{corollary}\label{c5} A   Banach space $ X $  has the $ p $-$ (DPrcP) $ if and only if
every bounded subset of $ X^{\ast} $ is a $ p $-Right set.\
\end{corollary}
\begin{remark}\rm\label{r3}
$ \rm{(i)} $ Every relatively weakly compact subset of  a dual Banach space is a $ p $-Right set, but the converse,
in general, is false.\ For example,  $ \ell_{1} $ has the  Schur property and so, $ \ell_{1} $ has the  $ p $-$ (DPrcP). $\ Hence
by Corollary \ref{c5},  the closed unit ball $ B_{\ell_{\infty}} $ of $ \ell_{\infty} $ is a $ p $-Right set, while it is not relatively weakly compact.\\
$ \rm{(ii)} $ Every $ p $-$ (V) $ set  in $ X^{\ast} $ is a $ p $-Right set, but the converse, in general, is false.\  For example, since $ \ell_{2} $ has the $ 2 $-$ (DPrcP),$ by Corollary \ref{c5}  the closed unit ball $ B_{\ell_{2}} $ of $ \ell_{2} $ is a $ 2 $-Right set, while it is not  $ 2 $-$ (V) $ set.\
\end{remark}
\begin{theorem}\label{t12}  A Banach space $ X $ has the $ (DPP_{p}) $ if and only if each  $ p $-Right set  in $ X^{\ast} $ is  a $ p $-$ (V) $ set.
\end{theorem}
\begin{proof}
Suppose  that $ (x_{n})_{n} $ is a weakly $ p $-summable sequence in $ X. $\  Therefore,  {\rm (\cite[Theorem 3.1]{ccl})} implies
 that $ \lbrace x_{n} :n\in\mathbb{N}\rbrace$ is a Dunford-Pettis set in $ X, $ since $ X \in (DPP_{p}).$\ Hence  every $ p $-Right set  in $ X^{\ast} $ is  a $ p $-$ (V) $ set.\\
Conversely, by Theorem \ref{t6},
it is enough to show that for each
Banach space $ Y,~ $ every pseudo weakly compact operator of order $ p $ is $ p $-convergent.\
For this purpose, let $ T\in PwC_{p}(X,Y) .$\ Theorem \ref{t11} implies that, $T^{\ast}(B_{Y^{\ast}})  $ is a $ p$-Right set.\ Thus by our hypothesis, $T^{\ast}(B_{Y^{\ast}}) $ is a $ p $-$ (V) $ set.\ Hence $ T:X\rightarrow Y $ is $ p $-convergent.\
\end{proof}
\begin{example}\rm\label{e1}
$ \rm{(i)} $
Hardy space $ H_{1} $ whenever $ 1 < p < 2 $ has the $ (DPPp). $\ Hence every $ p $-Right set in $ H^{\ast}_{1} $
 is a $ p $-$ (V ) $ set.\\
$ \rm{(ii)} $ If $  T$ is Tsirelson's space, then for $ 1 < p < \infty $ every $ p $-Right set in $ T^{\ast} $ is a $ p $-$ (V ) $ set, while
there exists $  p$-Right set in $ T^{\ast\ast} $ such that is not a $ p $-$ (V ) $ set.\\
$ \rm{(iii)} $ Let $ 1 < r < \infty. $\ If $ r^{\ast} $ denotes the conjugate number of $ r, $ then for $ p < r^{\ast} $ every every
$ p $-Right set in $ \ell_{r^{\ast}} $ is a $ p $-$ (V ) $ set.\\
\end{example}
\begin{theorem}\label{t13}
Let  the dual $ \mathcal{M^{\ast}} $ of a
closed subspace $ \mathcal{M}\subseteq \mathcal{U}(X, Y ) $ has the $ p $-$ (DPrcP) .$\ Then all of the point
evaluations $ \mathcal{M}_{1}(x) $ and $\widetilde{\mathcal{M}} _{1}(y^{\ast})$ are $ p $-Right sets.
\end{theorem}
\begin{proof}
Since $ \mathcal{M^{\ast}} $ has the $ p $-$ (DPrcP),$  by Theorem \ref{t3} the adjoint operator $ \phi_{x} ^{\ast}$
 is pseudo weakly
compact of order $ p. $\ Now, suppose that $ (y_{n}^{\ast})_{n} $ is a
$ p $-Right null sequence in $ Y^{\ast}. $\ It is clear that $ \displaystyle\lim_{n\rightarrow\infty} \Vert \phi_{x}^{\ast}(y_{n}^{\ast})) \Vert =0,$ for all $ x\in X. $\ On the other hand,
$$ \Vert \phi_{x}^{\ast}(y_{n}^{\ast})) \Vert=\sup\lbrace \vert \phi_{x}^{\ast}y_{n}^{\ast} (T ) \vert  : T\in B_{\mathcal{M}}  \rbrace=\sup\lbrace \vert  y_{n}^{\ast} (T(x))  \vert  : T\in B_{\mathcal{M}}  \rbrace .$$\
This shows that  $ \mathcal{M}_{1}(x) $ is a $ p $-Right set in $ Y, $ for all $ x\in X. $\ A similar proof shows
that $\widetilde{\mathcal{M}} _{1}(y^{\ast})$ is a  a $ p $-Right set in  $ X^{\ast}. $
\end{proof}

\begin{theorem}\label{t14}
Let $ X $ be a Banach space and $ 1\leq p_{1}<p_{2}\leq \infty. $\ The following statements are equivalent:\\
$\rm{(i)}$ For every Banach space $ Y, $ if  $ T : X\rightarrow Y $ is a pseudo weakly compact operator of order $ p_{1}, $ then $ T$   has a weakly $ p_{2} $-precompact (weakly $ p_{2} $-compact) adjoint,\\
$\rm{(ii)}$ Same as $\rm{(i)}$ with $ Y=\ell_{\infty} ,$\\
$\rm{(iii)}$ Every $  p_{1}$-Right subset of $ X^{\ast} $ is weakly $ p_{2} $-precompact (relatively weakly $ p_{2} $-compact).
\end{theorem}
\begin{proof}
We will show that  in the relatively weakly
$ p_{2} $-compact case.\ The other
proof is similar.\\
(i) $ \Rightarrow $ (ii) It is obvious.\\
(ii) $  \Rightarrow$ (iii) Let $  K$ be a $  p_{1}$-Right subset of $ X^{\ast} $ and let $ (x^{\ast}_{n})_{n} $
be a sequence
in $ K. $\  Define $ T:X\rightarrow \ell_{\infty} $ by $ T(x)=(x^{\ast}_{n}(x)) .$\
 Let $ (x_{n})_{n} $ be a $ p_{1} $-Right null sequence in $ X. $\ Since $ K $ is a $  p_{1}$-Right  set,
$$\lim_{n\rightarrow \infty}\Vert T(x_{n})\Vert=\lim_{n\rightarrow \infty}\sup_{m} \vert x^{\ast}_{m}(x_{n})\vert=0. $$
Therefore, $  T $ is pseudo weakly compact operator of order $ p_{1}. $\ Hence,  $ T^{\ast} $ is weakly $ p_{2} $-compact, and
$  (T^{\ast}(e^{1}_{n}))_{n}=(x^{\ast}_{n})_{n}$
has a weakly $ p_{2} $-convergent subsequence.\\
(iii) $ \Rightarrow $ (i) Let $ T:X\rightarrow Y $ be a  pseudo weakly compact operator of order $ p_{1}. $\ Then $ T^{\ast}(B_{Y^{\ast}}) $
is a $  p_{1}$-Right  subset of $ X. $\ Therefore $ T^{\ast}(B_{Y^{\ast}})  $ is relatively weakly $ p_{2} $-compact,
and thus $  T^{\ast}$ is weakly $ p_{2} $-compact.
\end{proof}

Let $ A $ and $ B $ be nonempty subsets of a Banach space $ X, $ we define  ordinary distance and non-symmetrized Hausdorff distance respectively, by
\begin{center}
$  d(A,B)=\inf\lbrace d(a,b):a\in A, B\in B \rbrace,   \hspace{.9 cm} \hat{d}(A,B) =\sup\lbrace d(a,B): a \in A\rbrace.$
\end{center}
Let $ X $ be a Banach space and $ K $ be a bounded subset of $ X^{\ast}. $ For $ 1\leq p\leq \infty, $
 we set
\begin{center}
$ \zeta_{p}(K)=\inf\lbrace \hat{d}(A,K) : K\subset X^{\ast}$ is a $ p $-Right set $ \rbrace. $
\end{center}
We can conclude that $  \zeta_{p}(K)=0 $ if and only if $K\subset X^{\ast}$ is a $ p $-Right set.\ Now, 
let $  K$ be a bounded subset of a Banach space $ X. $ The de Blasi measure
of weak non-compactness of $  K$ is defined by
\begin{center}
$\omega(K) =\inf\lbrace \hat{d}(K,A):\emptyset \neq A\subset X $ is weakly compact $ \rbrace. $
\end{center}
It is clear that  
  $ \omega(K) =0$ if and only if $ K $ is relatively weakly compact.\ For a bounded linear  operator
$ T : X \rightarrow Y , $ we denote  $ \zeta_{p} (T(B_{X})), \omega(T(B_{X}) $by  $ \zeta_{p} (T), \omega(T)$  respectively.\\

Note that every weakly compact operator is pseudo weakly compact of order $  p,$\ but in
general the converse is not true.\ For example,  the identity operator on $ \ell_{1} $
is pseudo weakly compact of order $ p ,$ while it is  not weakly compact.\
\begin{theorem}\label{t15}
Let $ X $ be a Banach space.\ The following statements are equivalent:\
$ \rm{(i)} $ $ X $ has the  $ p $-$  (SR)$ property.\\
$ \rm{(ii)} $ $ PwC_{p}((X,Y)= W(X,Y) ,$ for each Banach space $ Y. $\\
$ \rm{(iii)} $ $ PwC_{p}(X,\ell_{\infty})= W(X,\ell_{\infty}). $\\
$ \rm{(iv)} $ $ \omega (T^{\ast}) \leq \zeta_{p}(T^{\ast})$ for every bounded linear operator $  T$ from $  X$ into any Banach space $ Y. $\\
$ \rm{(v)} $ $ \omega (K)\leq \zeta_{p} (K)$ for every bounded subset $ K $ of $ X^{\ast} .$
\end{theorem}
\begin{proof}
 (i) $ \Rightarrow $ (ii)  Suppose that $ T \in PwC_{p}(X,Y) .$\ Theorem \ref{t11} implies that, $T^{\ast}(B_{Y^{\ast}})  $ is a $ p$-Right set.\ Since, $ X $ has the $ p $-$ (SR)$  property, $T^{\ast}(B_{Y^{\ast}} ) $ is
relatively weakly compact and so, $ T $ is weakly compact.\\
(ii) $ \Rightarrow $ (iii) It is trivial.\\
 (iii) $ \Rightarrow $ (i)   If   $X $ does not have the $ p $-$(SR)$ property, then there exists a $ p $-Right subset $ K $   in $ X^{\ast} $ so that it is not relatively weakly compact.\ So, there is a sequence
$ (x^{\ast}_{n})_{n}\subset K $ with no weakly convergent subsequence.\ Now, we show that the operator
$ T:X\rightarrow \ell_{\infty}$ defined by $ T (x) = (x^{\ast}_{n}(x)) $ for all $ x\in X $ is  pseudo weakly compact
of order $ p, $ but it is not weakly
compact.\ As $ (x^{\ast}_{n})_{n}\subset K  $ is a  $ p $-Right set, for every  $ p $-Right null sequence $ (x_{m})_{m}, $
 in $ X, $ we have $\displaystyle\lim_{m\rightarrow\infty} \Vert T(x_{m})\Vert=\displaystyle\lim_{m\rightarrow\infty}\sup_{n}\vert x^{\ast}_{n}(x_{m})\vert=0, $
hence, $ T\in PwC_{p} (X,\ell_{\infty}).$\ It is clear that $ T^{\ast}(\lambda_{n})=\displaystyle\sum_{n=1}^{\infty}\lambda_{n}x^{\ast}_{n} $  for every $ (\lambda_{n})_{n} \in
 (\ell_{\infty})^{\ast}.$\ If $ e^{1}_{n} $
 is the usual basis element in $ \ell_{1}, $
then $T^{\ast}(e^{1}_{n})=x^{\ast}_{n},  $
 for all $ n\in \mathbb{N}. $\ So  $ T^{\ast} $ is not a weakly compact operator and then $ T$  is not weakly compact.\\
 (i) $ \Rightarrow $ (iv)  It is obvious.\ (iv) $ \Rightarrow $ (i) It is immediate from Theorem \ref{t1}.\ The
equivalence of (i) $  \Leftrightarrow$ (v) is straightforward.
\end{proof}

\begin{corollary}\label{c6} If
$ X $ has the $ (DPP_{p}),$ then $ X $ has Pelczy$\acute{n} $ski's property $ (V ) $ of order $p $ if and only if  $ X $  has the $ p $-$ (SR) $ property.
\end{corollary}
\begin{proof}
 Suppose that $ T \in PwC_{p}(X,Y) .$\ Theorem \ref{t11}, implies that $ T^{\ast}(B_{Y^{\ast}}) $ is a $ p $-Right set in $ X^{\ast}. $\ As $ X\in (DPP_{p}), $ it follows
from  Theorem \ref{t12} that $ T^{\ast}(B_{Y^{\ast}}) $ is a $ p $-$ (V) $  set in $ X^{\ast}. $\ Since $ X $ has  Pelczy$ \acute{n} $ski's property $ (V ) $ of order $ p, $
$ T^{\ast} $ is weakly compact and so $ T $ is weakly compact.\ Hence, Theorem \ref{t15} implies that $X $ has the  $  p  $-$ (SR)$ property.
Conversely, If $ X $
has the $ p $-$ (SR)$ property, then $ X $ has Pelczy$\acute{n} $ski's property $ (V ) $ of order $ p.$\ Since  every $ p $-$ (V) $ set  in $ X^{\ast} $ is a $ p $-Right set.
\end{proof}
\begin{remark}\rm \label{r4}
$ \rm{(i)} $ It is clear that,  if  $ K $ is a infinite compact Hausdorff metric space, then the Banach space $ C(K) $  of  all continuous functions on $ K $  has Pelczy$\acute{n} $ski's property $ (V ) $  of order $p .$\ On the other hand $ C(K) $  has the $ (DPP_{p}).$\ Hence, $ C(K) $  has the  $ p $-$ (SR) $ property .\ In particular, $ c_{0} $ and $ \ell_{\infty} $ have the $ p $-$ (SR)$  property.\ However,  {\rm (\cite[Example 8]{pvwy})}
 shows that  $ \ell_{1} $ as a subspace
of $ \ell_{\infty} $ does not have the  $ (SR)$  property  and so, does not have the $ p $-$ (SR)$  property.\\
$ \rm{(ii)} $ It is clear that, every reflexive Banach space has the $ p $-$ (SR) $ property.\ But,  there exists a non reflexive Banach space with the $ p $-$ (SR) $ property.\ For example, $ c_{0} $ has the $ p $-$  (SR)$ property, while $ c_{0} $ is not reflexive space.
\end{remark}
 If  $ Y $ is a subspace of $ X^{\ast}, $ then we define $^{\bot}Y:=\lbrace x\in X: y^{\ast}(x)=0 ~~$ for all $~~~ y^{\ast}\in Y^{\ast} \rbrace  .$

\begin{corollary}\label{c7}  
$ \rm{(i)} $ If $ X $ has both  properties of $ p $-$ (DPrcP) $ and  $ p $-$ (SR),$ then $ X $ is a reflexive space.\\
$ \rm{(ii)} $ If $  X$ is an infinite dimensional non reflexive Banach space with the $ p$-Schur property, then $  X$
does  not  have the $ p $-$ (SR)$ property.\\
$ \rm{(iii)} $ If every separable subspace of $  X$ has  the $ p $-$ (SR) $ property,  then $ X  $ has the same property.\\
$ \rm{(iv)} $ If  $  X$ has the $ p $-$ (SR)$ property,  then every quotient space of $ X $ has the same property.\\
$ \rm{(v)} $ Let $Y $ be a reflexive subspace of $ X^{\ast}. $\ If $ ^{\bot}Y$  has the $ p $-$ (SR) $ property, then $ X $ has the $ p $-$ (SR) $ property.\
\end{corollary}
\begin{proof}
$ \rm{(i)} $ Suppose that $ X $ has the $ p $-$(DPrcP). $\ Therefore, the identity operator $ id_{X} : X \rightarrow X  $ is  pseudo weakly compact 
of order $ p. $\ As $ X$ has the $ p $-$ (SR)$ property, it follows from Theorem \ref{t15} that $ id_{X} $ is weakly compact  and so $ X $ is reflexive.\\
$ \rm{(ii)} $ Since $ X\in C_{p},$  the identity operator $ id_{X} : X \rightarrow X  $ is  $ p $-convergent and so, it is
pseudo weakly compact of order $ p. $\ It is clear that  $  id_{X} $ is not weakly compact.\ Theorem \ref{t15} implies that $ X $ does not have the $ p $-$ (SR)$ property.\\
$ \rm{(iii)} $ Let $ (x_{n})_{n} $ be a sequence in $ B_{X} $ and let $ Z = [x_{n} : n \in \mathbb{N}] $ be the closed linear
span of $ (x_{n})_{n} .$\ Since $  Z$ is a separable subspace of $ X, $ $ Z $ has the $ p $-$(SR)$ property.\ Now, let
$ T : X \rightarrow Y $ be a pseudo weakly compact operator
of order $ p. $\ It is clear that  $ T_{\vert Z} $ is pseudo weakly compact operator
of order $ p. $\ Therefore, Theorem \ref{t15}, implies that  $ T_{\vert Z} $ is weakly compact.\ Hence,
 there is a subsequence $ (x_{n_{k}} )_{k} $ of $  (x_{n})_{n} $ so that $ T(x_{n_{k}} ) $ is
weakly convergent.\ Thus $  T$ is weakly compact.\ Now an appeal to Theorem \ref{t15} completes the proof.\\
$ \rm{(iv)} $  Suppose that $  X$ has the  $ p $-$(SR)  $ property.\ Let $  Z$ be a quotient space
of $  X$  and $ Q: X\rightarrow Z $ be a quotient map.\ Let $ T:Z\rightarrow Y $ be a pseudo weakly compact operator  of order $ p. $\ It is clear that $ T\circ Q :Z\rightarrow Y $ is a pseudo weakly compact operator  of order $ p. $\ Hence
Theorem \ref{t15} implies that $ T\circ Q $ is weakly compact and so $ (T\circ Q)^{\ast} $ is weakly compact.\
Since $Q^{\ast} (T^{\ast}(B_{Y^{\ast}}) )$ is  weakly
compact and $ Q^{\ast} $ is an isomorphism, $T^{\ast}(B_{Y^{\ast}})  $
 is weakly compact.\ Hence, $ T $ is a weakly compact operator.\ Apply Theorem \ref{t15}.\\
 $ \rm{(v)} $  By  {\rm (\cite[Theorem 1.10.6]{m})},
 there exists a quotient map  $ Q: X^{\ast}\rightarrow \frac{X^{\ast}}{Y} $ and a surjective isomorphism $ i:\frac{X^{\ast}}{Y}\rightarrow ~(^{\bot}Y) ^{\ast}$ such that $ i\circ Q:X^{\ast}\rightarrow   ~(^{\bot}Y)^{\ast} $ is $ w^{\ast} $-$ w^{\ast} $ continuous.\ So, there is $ S: ~^{\bot}Y \rightarrow X $
 with $ S^{\ast}=i\circ Q. $\ Hence, for any pseudo weakly compact operator of order $ p,~~~ $  $ T:X\rightarrow Z, $
the operator $ T\circ S: ~^{\bot}Y \rightarrow Z$ is pseudo weakly compact  of order $ p, $ that must
be weakly compact; hence,  $ S^{\ast}\circ T^{\ast} =i\circ Q\circ T^{\ast} $ is also weakly compact,
this in turn gives that $ Q\circ T^{\ast} $ must be weakly compact, since $  i$ is a surjective
isomorphism.\ Therefore
$ T^{\ast} $ and so $ T $ is weakly compact.\ The Theorem \ref{t15} completes the proof.\
\end{proof}
Notice that every relatively norm compact subset of $ X^{\ast} $ is  $ p $-Right set.\
But, the converse is not necessarily correct.\ For example, for each $ 1 < p < \infty $ and for each $ 1 < r < p^{\ast}, $
the identity operator $ id_{r} $ on $ \ell_{r} $ is pseudo weakly compact  of order $ p $ and hence the unit ball $ B_{\ell_{r^{\ast}}} $
of $ \ell_{r^{\ast}} $ is a $ p $-Right set.\\
In the following, we give a necessary and sufficient
condition that every $ p $-Right subset of $ X^{\ast} $ is relatively norm compact.\ 
Since the proof of the following result is similar to the proof of Theorem \ref{t15}, we omit its proof.\
\begin{theorem}\label{t16}
Let  $ X$ be a Banach space.\ The following
are equivalent:\
$ \rm{(i)} $   $PwC_{p} (X,Y) =K(X,Y),$ for every Banach space $ Y.$\\
$ \rm{(ii)} $ Same as $ \rm{(i)} $ with $ Y=\ell_{\infty}. $\\
$ \rm{(iii)} $ Every $ p $-Right subset of $ X^{\ast} $ is relatively  norm compact.
\end{theorem}

It is clear that, if we speak about $  \mathcal{U} (X,Y)$ or its linear subspace $\mathcal{M},  $ then the related norm is the ideal norm $ \mathcal{A}(.) $
while, the operator norm $ \Vert .\Vert $ is applied when the space is a linear subspace of $L(X,Y) . $\
Now, we obtain some conditions for which the point evaluations $ \mathcal{M}_{1}(x) $ and $\widetilde{\mathcal{M}} _{1}(y^{\ast})$ are relatively weakly compact for all $ x \in X $ and all $ y^{\ast} \in Y^{\ast}. $\
\begin{theorem}\label{t17}
Suppose that $ X^{\ast\ast} $ and $ Y^{\ast} $ have the $ p $-$ (SR)$ property,  and $\mathcal{M}\subseteq \mathcal{U} (X,Y) $ is a closed subspace.\ If the natural restriction operator $ R: \mathcal{U} (X,Y)^{\ast}\rightarrow \mathcal{M^{\ast}} $
 is a  pseudo weakly
compact operator of order $ p, $ then all of the point evaluations $ \mathcal{M}_{1}(x) $ and $\widetilde{\mathcal{M}} _{1}(y^{\ast})$ are relatively weakly compact.\
\end{theorem}
\begin{proof}
It is enough to show
that $ \phi_{x} $  and $ \psi_{y^{\ast}} $ are weakly
compact operators.\ For this purpose suppose that $ T\in \mathcal{U} (X,Y).$\ Since $ \Vert T \Vert\leq \mathcal{A}(T), $ it is not difficult to show that,
the operator $ \psi:X^{\ast\ast}\widehat{\bigotimes}_{\pi}Y^{\ast}\rightarrow \mathcal{U} (X,Y)^{\ast} $  which is defined by
$$ \vartheta\mapsto tr(T^{\ast\ast}\vartheta) =\displaystyle \sum_{n=1}^{\infty}\langle T^{\ast\ast} x^{\ast\ast} _{n}, y_{n}^{\ast}  \rangle$$ is linear and continuous, where $\vartheta=\displaystyle \sum_{n=1}^{\infty} x_{n} ^{\ast\ast}\bigotimes y_{n}^{\ast}.$\ Fix now an arbitrary element $ x\in X $ and define the operator $ U_{x}:Y^{\ast}\rightarrow X^{\ast\ast}\widehat{\bigotimes}_{\pi}Y^{\ast}$ by $ U_{x}(y^{\ast})=x\bigotimes y^{\ast}. $\
It is clear that the operator $ \phi_{x}^{\ast}= R\circ \psi \circ U_{x}$ is a pseudo weakly
compact  operator of order $ p. $\
Since $ Y^{\ast}$ has the $ p $-$ (SR) $ property, we conclude that
$ \phi_{x}^{\ast}$ is  a weakly
compact operator.\ Hence, $ \phi_{x} $  is  weakly
compact.\ Similarly, we can see that $ \psi_{y^{\ast}} $ is  weakly
compact.\
\end{proof}
\begin{theorem}\label{t18}
 Suppose that $L_{w^{\ast}}(X^{\ast} ,Y)= K_{w^{\ast}}(X^{\ast} ,Y) .$\ If $ X$ and $ Y $ have  the $ p $-$ (SR)$ property, then $ K_{w^{\ast}}(X^{\ast} ,Y) $ has the same property.
\end{theorem}
\begin{proof}
Suppose $ X $ and $ Y $ have the $ p $-$ (SR) $ property.\  Let $  H$ be a
$ p $-Right subset of $  K_{w^{\ast}}(X^{\ast} ,Y). $\
For fixed $ x^{\ast}\in X^{\ast}, $ the map $ T\mapsto T(x^{\ast}) $ is a bounded
operator from $  K_{w^{\ast}}(X^{\ast} ,Y) $ into $ Y. $\ It is easily verified that continuous linear images of $ p $-Right sets are $ p $-Right sets.\ Therefore, $  H(x^{\ast}) $
 is a $ p $-Right subset of $ Y, $ hence relatively
weakly compact.\ For fixed $ y^{\ast}\in Y^{\ast} $ the map $ T\mapsto T^{\ast}(y^{\ast}) $
 is a bounded  linear operator from
$  K_{w^{\ast}}(X^{\ast} ,Y) $ into $ X. $\ Therefore, $ H^{\ast}(y^{\ast}) $
 is a $ p $-Right  subset of $ X, $ hence  relatively weakly compact.\ Then,
by {\rm (\cite[Theorem 4.\ 8]{g17})}, $ H $ is  relatively weakly compact.
\end{proof}
As an application of Theorem \ref{t18}, we obtain a sufficient condition for the $ p $-$  ( SR)  $ property of the compact operators space.
\begin{corollary}\label{c8} Suppose that   $L(X ,Y)= K(X ,Y) .$\ If $ X^{\ast} $  and $ Y $ have the $ p $-$ (SR)$   property, then $  K(X, Y ) $ has the same property.
\end{corollary}
Ghenciu in \cite{g8}
investigated whether
the projective tensor product of two Banach spaces $ X $ and $ Y $ has the sequentially Right
property when $  X$ and $  Y$ have this property.\  Here,
we investigate  whether
the projective tensor product and injective tensor product of two Banach spaces $ X $ and $ Y $ has the $ p $-sequentially Right
property.\ 
\begin{theorem}\label{t19}
Suppose that $ X$ has the $ p $-$ (SR) $ property and $ Y $ is a reflexive space.\ If $ L(X,Y^{\ast}) = K(X,Y^{\ast}),$
 then
 $ X \widehat{\bigotimes}_{\pi} Y $ and $ X \widehat{\bigotimes}_{\varepsilon} Y $ have the $ p $-$(SR)$ property.\
\end{theorem}
\begin{proof} 
We only prove the result for the projective tensor product. The other proof is similar.\
 Let $ H $ be a $ p $-Right subset of $ (X \widehat{\bigotimes}_{\pi} Y)^{\ast}\simeq L(X,Y^{\ast}). $\ Let $ (T_{n}) $ be an arbitrary sequence in $ H $ and let $ x\in X. $\ We first show that $ \lbrace T_{n}(x) : n\in \mathbb{N} \rbrace$
is a $ p $-Right subset of $ Y^{\ast}. $\ Let $ (y_{n})_{n} $ be a $ p $-Right null sequence sequence in $ Y. $\  For each $ n\in \mathbb{N},  $
$$  \langle  T_{n} (x) ,y_{n} \rangle=\langle T_{n}, x\otimes y_{n}\rangle .$$
We claim that $  (x \otimes y_{n})_{n}$ is a  $ p $-Right null sequence  in $ X \widehat{\bigotimes}_{\pi} Y. $\ If $ T\in (X \widehat{\bigotimes}_{\pi} Y)^{\ast}\simeq L(X,Y^{\ast}),  $ then
$$ \vert\langle T, x\otimes y_{n}\rangle\vert=\vert \langle T(x), y_{n} \rangle\vert \in \ell_{p},$$
since $ (y_{n}) _{n}$ is weakly $ p $-summable.\ Thus $  (x \otimes y_{n})_{n}$ is weakly $ p $-summable in $ X \widehat {\bigotimes}_{\pi} Y. $\
 Let $ (A_{n})_{n} $ be a weakly null sequence in $ (X \widehat {\bigotimes}_{\pi} Y)^{\ast}\simeq L(X,Y^{\ast}). $\ Since the map $ \phi_{x} : L(X,Y^{\ast})\rightarrow Y^{\ast},$
 $ \phi_{x}(T) =T(x)$
 is linear and bounded, $ (A_{n}(x)) $ is weakly null in $ Y^{\ast}. $\ Therefore
$$ \vert \langle A_{n}, x\otimes y_{n}\rangle \vert =\vert \langle A_{n}(x), y_{n}\rangle\vert\rightarrow 0,$$
since $ (y_{n})_{n} $ is a Dunford-Pettis sequence  in $ Y. $\
Therefore,  $ (x  \otimes y_{n})_{n} $ is a Dunford-Pettis sequence in $ X \widehat {\bigotimes}_{\pi} Y. $\ Hence, $ (x\otimes y_{n})_{n} $ is  $ p $-Right null in $ X \widehat {\bigotimes}_{\pi} Y$ and so,
$ \lbrace T_{n}(x) :n\in  \mathbb{N}\rbrace$ is a $ p $-Right set in $ Y^{\ast}. $\ Therefore,  
$ \lbrace T_{n}(x) :n\in  \mathbb{N}\rbrace$ is a relatively weakly compact.\ Now, let $ y\in Y^{\ast\ast} =Y$ and $ (x_{n})_{n} $ be a 
$ p $-Right null sequence in $ X. $\ An argument similar to the one above shows that $ (x_{n}\otimes y)_{n} $ is a p-Right null sequence in
$ X \widehat {\bigotimes}_{\pi} Y. $\ Hence,
\begin{center}
$  \vert \langle  T^{\ast}_{n}(y), x_{n}  \rangle \vert=\vert \langle  T_{n}(x_{n}) ,y\rangle\vert=\vert  \langle   T_{n}, x_{n}\otimes y \rangle\vert\rightarrow 0,$
\end{center}
since $ (T_{n})_{n} $ is a $ p $-Right set.\ Therefore $ \lbrace T^{\ast}_{n} (y):n\in \mathbb{N}\rbrace$ is a $ p $-Right subset of $ X^{\ast}.$\ Hence, $ \lbrace T^{\ast}_{n} (y):n\in \mathbb{N}\rbrace$  is relatively weakly compact.\ Theorem 3 of \cite{g2} implies that $ H $ is relatively weakly compact.\
\end{proof}
For every  $ n\in \mathbb{N},$  we
denote the canonical projection from $ (\displaystyle\sum_{n=1}^{\infty}\oplus X_{n})_{\ell_{r}} $ 
 into $ X_{n} $ by $ \pi_{n} .$\ Also, we denote the canonical
projection from $ (\displaystyle\sum_{n=1}^{\infty}\oplus X^{\ast}_{n})_{\ell_{r^{\ast}}} $
 onto $ X^{\ast}_{n} $
 by $ P_{n} .$\\
As an immediate consequence of the Theorem \ref{t12} and {\rm (\cite[Theorem 3.1]{ccl1})}, we obtain the following result:
\begin{theorem}\label{t20}
Let $ 1 < p<\infty $ and $ (X_{n})_{n}  $ be a sequence of Banach spaces with $ (DPP_{p}) $ and let $ X=(\sum_{n=1}^{\infty}\oplus X_{n})_{\ell_{p}}.$\
  The following are equivalent
for a bounded subset $  K$ of $ X^{\ast}: $\\
$ \rm{(i)} $ $ K $ is a $ p^{\ast} $-Right set.\\
$ \rm{(ii)} $ $ P_{n} (K)$ is a $ q $-Right set for each $ n\in \mathbb{N} $ and
$$  \lim_{n\rightarrow\infty}\sup\lbrace \sum_{k=n}^{\infty}\Vert P_{k}x^{\ast}\Vert^{p^{\ast}}    : x^{\ast}\in K\rbrace=0.$$
\end{theorem}
\begin{theorem}\label{t21}
 Let $ 1 < p<\infty $ and $ (X_{n})_{n}  $  be a sequence of Banach spaces.\ If $ X=(\displaystyle\sum_{n=1}^{\infty}\oplus X_{n})_{\ell_{p}} $
and $ 1 \leq q < p^{\ast} ,$ 
then a bounded subset $ K $ of $ X^{\ast} $ is a
$ q $-Right set if and only if each $ P_{n}(K) $ is.
\end{theorem}
\begin{proof}  It is easily verified that
continuous linear images of $ q $-Right set is $  q$-Right set.\ Therefore, we only prove the sufficient part.\ Assume that $ K $ is not a $ q $-Right set.\ Therefore, there exist $ \varepsilon_{0}>0, $ a $ q $-Right null sequence $ (x_{n})_{n} $ in $ X $ and a sequence 
$ (x^{\ast}_{n})_{n} $ in $ K  $ such that 
\begin{center}

$ \vert \langle x^{\ast}_{n}, x_{n}     \rangle \vert=\vert \displaystyle\sum _{k=1}^{\infty}  \langle  P_{k} x^{\ast}_{n} , \pi_{k} x_{n} \rangle  \vert   > \varepsilon_{0},~~~~~~~~~~~~~~n =1,2,3,...~~~~~~~~~~$  $ ~~~~~~~~~~~~~(\ast )$
\end{center}
By the assumption, we obtain
\begin{center}
$\displaystyle\lim_{n\rightarrow \infty} \vert  \langle P_{k} x^{\ast}_{n} ,\pi_{k} x_{n}      \rangle \vert=0 , ~~~~~~~~ k =1,2,3,...~~~~$ $ ~~~~~~ (\ast\ast) $
\end{center}

By induction on $ n $ in $ (\ast) $ and $  k$ in $ (\ast\ast), $ there exist two 
 strictly increasing sequences $ (n_{j})_{j} $ and $ (k_{j} )_{jn}$ of positive integers such that 
\begin{center}
$ \vert \displaystyle\sum_{k=k_{j-1}+1}^{k_{j}} \langle P_{k} x^{\ast}_{n_{j}} , \pi_{k} x_{n_{j}} \rangle\vert >\frac{\varepsilon_{0}}{2}, ~~~~~~j=1,2,3,...$
\end{center}
For each $ j = 1, 2, ..., $ we consider $ y_{j}=x_{n_{j}} $  and  $ y^{\ast}_{j} \in X^{\ast}$
by

\begin{equation*}
P_{k}y^{\ast}_{j}=
\begin{cases}
P_{k_{j}}x^{\ast}_{n_{j}} & \text{if }  k_{j-1}+1 \leq k\leq k_{j},\\
0 & \text{otherwise }.
\end{cases}
\end{equation*}
 It is clear that $ ( y_{j})_{j} $ is a $q $-Right null sequence in $ X $ such that 
\begin{center}
$ \vert\langle y^{\ast}_{j} ,y_{j}\rangle\vert=\vert \displaystyle\sum_{k=k_{j-1}+1}^{k_{j}} \langle P_{k} x^{\ast}_{n_{j}} , \pi_{k} x_{n_{j}} \rangle\vert >\frac{\varepsilon_{0}}{2}, ~~~~~~j=1,2,3,... $
\end{center}

Since the sequence $ ( y^{\ast}_{j})_{j} $
has pairwise disjoint supports, Proposition 6.4.1 of \cite{AlbKal} implies that $ ( y^{\ast}_{j})_{j} $
is equivalent to the unit vector basis $ ( e^{p^{\ast}}_{j})_{j} $ of $ \ell_{p^{\ast}} .$\ Suppose that $ R $ is an isomorphic
embedding from $ \ell_{p^{\ast}} $ into $ X^{\ast} $ such that $ R( e^{p^{\ast}}_{j})=y^{\ast}_{j}  (j = 1, 2, ...).$\
 Now, let $ T $ be an any
operator from $\ell_{q^{\ast}}  $ into $ X. $\ By Pitt’s Theorem \cite{AlbKal}, the operator $ T^{\ast}R $ is compact
and hence the sequence $ (T^{\ast}(y^{\ast}_{j}))_{j}=(T^{\ast}R(e^{\ast}_{j}))_{j} $
is relatively norm compact.\ Hence, Theorem 2.3 of \cite{ccl1}  implies that the sequence $ (y^{\ast}_{j}) _{j}$ is a $ q $-$ (V) $ set and so is a
 $ q $-Right  set.\ Since $ (y_{j})_{j} $
is $ q $-Right null, we have
\begin{center}
$  \vert \langle  y^{\ast}_{n}, y_{n}\rangle\vert\leq \sup_{j}\vert \langle  y^{\ast}_{j}, y_{n}\rangle\vert\rightarrow 0$ as $  ~~~~~~n\rightarrow \infty, $
\end{center}
which is a contradiction.  
\end{proof}

\begin{theorem}\label{t22}
 Let $ (X_{n})_{n}  $  be a sequence of Banach spaces.\ If  $ 1 < r <\infty $ and $ 1 \leq p <\infty ,$
then   each  $ X_{n} $ has the $ p $-$ (SR) $ property if and only if $X=(\displaystyle\sum_{n=1}^{\infty}\oplus X_{n})_{\ell_{r}} $ has the same property. 
\end{theorem}
\begin{proof} Corollary \ref{c7} shows that if $ X $ has the $ p$-$ (SR) $ property,  then each $ X_{n} $ has the $ p $-$ (SR)$ property.\ Conversely, 
let $  K$ be a $ p $-Right subset of $ X^{\ast}. $\ It is clear that each $ P_{n}(K) $ is
also a $ p $-Right set.\ Since $ X_{n} $ has the  $ p $-$(SR) $ property for each $ n\in\mathbb{N}, $ each $ P_{n}(K) $ is relatively weakly compact.\ It
follows from Lemma 3.4 \cite{ccl1}  that $  K$ is relatively weakly compact.\
\end{proof}

 The concept of the weak sequentially Right property introdued by Ghenciu \cite{g8} as follows:``
A Banach space X has the weak sequentially Right (in short $ X \in (wSR) $) property if every Right subset of $ X^{\ast} $ is weakly precompact".\

\begin{definition}\label{d2}  
 A  Banach space  $ X $ has the weak $ p $-sequentially Right property  ( in short $ X$ has the $ p $-$ (wSR) $), if every $ p $-Right set  in $ X^{\ast}$ is  weakly  precompact.\
\end{definition}
The weak $ \infty $-sequentially Right property is precisely the weak sequentially Right property.\ It is clear that if $ X $ has the $ p $-$ (SR)$ property,  then $ X $ has the $ p $-$ (wSR) $ property.\ The following result
give an operator characterization from the class of  $ p $-Right sets which
are  weakly precompact.\ Since the proof of the following result is similar to the proof of Theorem \ref{t15}, we omit its proof.
\begin{theorem}\label{t23}
Let  $ X$ be a Banach space.\ The following assertions
are equivalent:\
$ \rm{(i)} $ $ X $ has the  $ p $-$ (wSR) $ property,\\
$\rm{(ii)} $ If $ T\in PwC_{p}(X,Y) $ for every Banach space $ Y, $
then
$ T^{\ast}  $ is weakly precompact,\\
$ \rm{(iii)} $ Same as $ \rm{(i)} $ with $ Y=\ell_{\infty}. $\
\end{theorem}

\begin{corollary}\label{c9} Let  $ X  $ and $ Y $ be two  Banach space.\ Then, the following statements hold:\
$ \rm{(i)} $ If $ X $ has the $ p $-$ (wSR) $ property, then every quotient space of $ X $ has this property,\\
$ \rm{(ii)} $ If $  X$  have both the $ p $-$ (DPrcP) $ and $ p $-$ (wSR) $ property, then $ X^{\ast} $ contains no copy of $ \ell_{1},$\\
$ \rm{(iii)} $ Suppose that $ L(X,Y^{\ast}) =K(X,Y^{\ast}).$\ If $ X \widehat{\bigotimes}_{\pi} Y $ has the $ p $-$ (wSR) $ property, then $  X$ and $ Y $ have this property and at least one of them does not contain $ \ell_{1}, $\\
$ \rm{(iv)} $ Suppose that $ X $ has the $ p $-$ (wSR) $ property and $ Y $ is a Banach space.\ If $ T\in PwC_{p}(X,Y) ,$ then $ T $ is weakly precompact,\\
$ \rm{(v)} $ If $X\in  (DPP_{p}), $ then $ X^{\ast} $ has the  $ p $-$ (wSR) $ property.
\end{corollary}
\begin{proof}
$ \rm{(i)} $ Suppose that $  X$ has the weak $ p $-sequentially Right  property.\ Let $ Z $ be a quotient space
of $ X $ and $  Q : X \rightarrow Z $ be a quotient map.\ Let $ T : Z \rightarrow Y $ be a pseudo
weakly compact operator of order $ p. $\ Then $ T\circ Q : X \rightarrow Y $ is a pseudo weakly compact operator of order $ p, $ and
so $ (T\circ Q)^{\ast} $ is weakly precompact by Theorem \ref{t23}.\ Since $ Q^{\ast}T^{\ast}(B^{\ast}_{Y}) $ is weakly
precompact and $ Q^{\ast} $ is an isomorphism, $ T^{\ast}(B^{\ast}_{Y})  $ is weakly precompact.\ Apply
Theorem \ref{t23}.\\
$ \rm{(ii)} $ Suppose that $  X$ has the $ p $-$ (DPrcP) $ and the weak $ p $-sequentially Right  property.\ Then the identity operator $ id_{X} : X \rightarrow X $ is a pseudo weakly compact  of order $ p .$\ Thus  Theorem
\ref{t23} implies that $ id_{X^{\ast}}: X^{\ast}\rightarrow X^{\ast} $
is weakly precompact.\ Hence $ X^{\ast} $ contains no copy of $ \ell_{1}, $ by
Rosenthal’s $ \ell_{1} $-theorem.\\
$ \rm{(iii)} $ Suppose that $ X \widehat{\bigotimes}_{\pi} Y \in (wSR)_{p}. $\ It is clear that $  X$ and $ Y $ have the $ (wSR)_{p} $ property, since the  weak $ p $-sequentially Right property is inherited by quotients.\ We will show that $\ell_{1}\hookrightarrow X  $ or $\ell_{1}\hookrightarrow Y.  $\ Suppose
that $\ell_{1} \not \hookrightarrow X  $ and $\ell_{1} \not \hookrightarrow Y. $\ Hence $ L_{1} \hookrightarrow X^{\ast}$ {\rm (\cite[p.\ 212]{di1})}.\
Also, the Rademacher functions span $ \ell_{2} $ inside of $ L_{1}, $ and thus $ \ell_{2} \hookrightarrow X^{\ast}.$\ Similarly
 $ \ell_{2} \hookrightarrow Y^{\ast}.$\ Then $ c_{0} \hookrightarrow  K(X, Y ^{\ast}) $ {\rm (\cite[Theorem 3]{e0})}.\
Thus $ \ell_{1} $ is complemented in $ X \widehat{\bigotimes}_{\pi} Y $
 {\rm (\cite[Theorem 10]{di1})}, which is a contradiction.\ Since $ \ell_{1}\not \in  (wSR)_{p} .$\\
 $ \rm{(iv)} $ Suppose that $  X\in (wSR)_{p}$ and  $ T\in PwC_{p}(X,Y). $\ Then $ T^{\ast} $ is weakly precompact by Theorem \ref{t23}.\ Now, we apply
Corollary 2 of \cite{bp} to complete the proof.\\
$ \rm{(v)} $ Since $ X $
has the $ (DPP_{p}), $ every weakly $ p $-summable sequence in $ X $ is Dunford-Pettis.\ Then every $ p $-Right subset
of $ X^{\ast} $ is a Dunford-Pettis  set, and thus is weakly precompact.\
\end{proof}
\section{ $ p $-Sequentially Right$ ^{\ast} $  property  on Banach spaces}

In this section by presenting  a new class of subsets of Banach spaces which are called $  p$-Right$ ^{\ast} $ sets, we introduce  two Banach space properties, the
$ p $-sequentially Right$ ^{\ast} $  and the  weak $  p$-equentially Right$ ^{\ast} $  properties.\ Then we obtain some characterizations of these sets and
properties.\
\begin{definition}\label{d3} 
A bounded subset $ K $ of a Banach space $ X $ is said $ p $-Right$ ^{\ast}$ set, if for
every $ p $-Right null sequence $ (x^{\ast}_{n})_{n} $ in $ X^{\ast} $ it follows:
$$ \lim_{n} \sup_{x\in K}\vert x_{n}^{\ast}(x) \vert=0.$$\
\end{definition}
The $ \infty $-Right$ ^{\ast} $ sets are precisely the Right$ ^{\ast} $ sets.\
It is clear that, if $ 1\leq p_{1} <p_{2} \leq\infty, $ then every $ p_{2} $-Right$ ^{\ast} $ set in $ X^{\ast} $ is a $ p_{1} $-Right$ ^{\ast} $ set.\ In particular, for $ 1\leq p<\infty $ every Right$ ^{\ast} $ set is a $ p $-Right$ ^{\ast} $ set.\\
\begin{theorem}\label{t24} Let $ T :Y\rightarrow X $ be abounded linear operator.\ Then $ T ^{\ast} $ is pseudo weakly compact of order $ p$ if and only if $ T $ maps bounded subsets of $ Y $ onto $ p$-Right$ ^{\ast} $ subsets of $ X. $
\end{theorem}
\begin{proof}
Let $ (x^{\ast}_{n})_{n} $ be a $ p $-Right null sequence in $ X. $\ Then the equalities:\
$$\Vert T^{\ast}(x^{\ast}_{n})\Vert=\displaystyle\sup_{y\in B_{Y}}\vert T(y)(x^{\ast}_{n})\vert=\displaystyle\sup_{y\in B_{Y}}\vert y(T^{\ast}(x^{\ast}_{n}))\vert$$
deduces the proof.
\end{proof}
\begin{corollary}\label{c10} Let $ X $ be a Banach space.\ Then $ X^{\ast} $ has the $ p $-$ (DPrcP) $ if and only if every bounded subset of $ X $ is a $ p $-Right$ ^{\ast} $ set.\
\end{corollary}

\begin{definition}\label{d4}
$ \rm{(i)} $ A Banach space $ X $ has the $ p $-sequentially Right$ ^{\ast} $ property  (in short $ X$ has the $ p $-$(SR^{\ast})$ property), if every $ p $-Right$ ^{\ast} $ set is relatively weakly
compact.\\
$ \rm{(ii)} $ A Banach space $ X $ has the weak sequentially Right$ ^{\ast} $ property of order $ p $ (in short $ X$ has the $ p $-$(wSR^{\ast})$ property), if every $ p $-Right$ ^{\ast} $ set is weakly
precompact.
\end{definition}
As an immediate consequence of  Definitions $ \rm\ref{d3} $ and $\rm{\ref{d4}} ,$ we can  obtain the following results:\
 \begin{proposition}\label{p7}
Let $ X $ be a Banach space.\ The following statements hold:\\
$ \rm{(i)} $ If $ X $ has the $ p $-$ (SR) $ property, then $ X^{\ast} $ has the $ p $-$ (SR^{\ast})$ property.\\
$ \rm{(ii)} $ If $ X^{\ast} $ has the $ p $-$ (SR) $ property, then $ X $ has the $ p $-$ (SR^{\ast}) $ property,\\
$ \rm{(iii)} $ If $X$ has the $ p $-$ (wSR) $ property, then $ X^{\ast} $ has the $ p $-$ (wSR^{\ast}) $ property,\\
$ \rm{(iv)} $ If $X^{\ast} $ has the $ p $-$ (wSR) $ property, then $X $ has the $ p $-$ (wSR^{\ast})$ property,\\
$ \rm{(v)} $ Every $ p $-$ (V^{\ast}) $ set is a $ p $-Right$ ^{\ast} $set,\\
$ \rm{(vi)} $ If $ X $ has the $ p $-$ (SR^{\ast}) $ property, then  $ X $ has the $ p $-$ (V^{\ast}) $ property.
\end{proposition}
\begin{theorem}\label{t25} 
$ \rm{(i)} $ Let  $ Y$ be a reflexive subspace of $ X. $\ If $ \frac{X}{Y} $ has the $ p $-$ (SR^{\ast}) $ property,
then $ X $ has the same property.\\
$\rm{(ii)}$ Let $ Y $ be a subspace of $ X$ not containing copies of $ \ell_{1} .$\
If $ \frac{X}{Y}$ has the $ p $-$ (wSR^{\ast}) $ property, then $ X $ has the same property.\
\end{theorem}
\begin{proof}
We only prove $ \rm{(i)} .$\ The proof of $ \rm{(ii)} $ is similar.\\
$ \rm{(i)} $ Let $ Q : X \rightarrow \frac{X}{Y} $ be the quotient map.\ Let $ K$ be a $ p $-Right$ ^{\ast}$ set in $ X $ and $ (x_{n})_{n} $ be
a sequence in $ K. $\ Then $ (Q(x_{n}))_{n} $ is a $ p $-Right$ ^{\ast} $ set in $ \frac{X}{Y},$ and thus relatively weakly
compact.\ By passing to a subsequence, suppose $ (Q(x_{n}))_{n} $ is weakly convergent.\ By
{\rm (\cite[Theorem 2.7]{gh})}, $ (x_{n})_{n} $ has a weakly convergent subsequence.\ Thus $ X$ has the $ p $-$ (SR^{\ast}) $ property.
\end{proof}
Let $ K $ be a bounded subset of $ X. $\ For $ 1\leq p\leq \infty, $
 we set
\begin{center}
$\vartheta_{p}(K)=\inf\lbrace \hat{d}(A,K) : K\subset X^{\ast}$ is a $ p $-Right$ ^{\ast} $ set $ \rbrace. $
\end{center}
We can conclude that $ \vartheta_{p}(K)=0 $ if and only if $K\subset X$ is a $ p $-Right$ ^{\ast} $ set.\   
For a bounded linear  operator
$ T : X \rightarrow Y , $ we denote  $ \vartheta_{p} (T(B_{X})) $ by  $ \vartheta_{p} (T).$\\
The following result shows that  $ p $-sequentially Right$ ^{\ast} $ property is
automatically quantitative in some sense.
\begin{theorem}\label{t26}
Let $ X $ be a Banach space.\ The following statements are equivalent:\\
$\rm{(i)}$ For every Banach space $ Y, $ if $ T : Y \rightarrow X $ is an operator such that $ T^{\ast} $ is a pseudo weakly compact operator of order $ p, $
then $ T $ is weakly compact,\\
$\rm{(ii)}$ Same as $\rm{(i)}$ with $ Y=\ell_{1} ,$\\
$\rm{(iii)}$ $ X$ has the  $ p $-$(SR^{\ast}) $ property,\\
$ \rm{(iv)} $ $ \omega(T^{\ast})\leq \vartheta_{p}(T^{\ast}) $ for every operator $  T$ from $  X$ into any Banach space $ Y, $\\
$ \rm{(v)} $ $ \omega(K) \leq \vartheta_{p}(K)$ for every bounded subset $  K$ of $ X. $
\end{theorem}
\begin{proof}
(i ) $ \Rightarrow $ (ii ) It is obvious.\\
(ii ) $ \Rightarrow $ (iii ) Let $ K $ be a $ p $-Right$ ^{\ast} $ subset of $ X$ and let $ (x_{n}) $ be a sequence in $ K. $\ Define $ T:\ell_{1}\rightarrow X $
by $ T(b)=\sum_{i} b_{i}x_{i}.$\ It is clear that that $ T^{\ast}(x^{\ast})=(x^{\ast}(x_{i}))_{i}. $\
Let $ (x_{n}^{\ast})_{n} $
be a $ p $-Right null
sequence in $ X^{\ast}. $\ Since $ K $ is a $ p $-Right$ ^{\ast} $ set,
$$\lim_{n} \Vert T^{\ast}(x^{\ast}_{n}) \Vert =\lim_{n}\sup _{i}\vert x_{n}^{\ast}(x_{i}) \vert=0.$$
Therefore $ T^{\ast}\in PwC_{p}(X^{\ast}, \ell_{\infty})$ and thus $ T$ is weakly compact.\ Hence, $ (T(e_{n}^{1}))_{n}=(x^{\ast}_{n})_{n} $
has a weakly convergent subsequence.\\
(iii ) $ \Rightarrow $ (i ) Let $ T : Y \rightarrow X $ be an operator such that $ T^{\ast} $ is a pseudo weakly compact operator of order $ p .$\ Let $ (x^{\ast}_{n})_{n} $
be a $ p $-Right null sequence in $ X^{\ast}. $\ If $ y\in Y, $ then $ \vert x_{n}^{\ast}(T(y)) \vert\leq \Vert T^{\ast}(x^{\ast}_{n}) \Vert\rightarrow 0. $\
Therefore $ T(B_{Y}) $
is a $ p $-Right$ ^{\ast} $ subset of $ X. $\ Hence $ T(B_{Y}) $ is weakly compact, and thus $ T$ is weakly
compact.\\
 The equivalence of (iii) $  \Leftrightarrow$ (iv)  and  (iii) $ \Leftrightarrow $  (v)   are straightforward.
\end{proof}
\begin{corollary}\label{c11} If $  X ^{\ast} $ has the  $ p $-$ (DPrcP)$ and $ Y$ has the  $ p $-$ (SR^{\ast}) $  property, then $ L(X,Y) =W(X,Y).$\
\end{corollary}
\begin{proof}
 Let $ T\in L(X,Y) $ and $ (y^{\ast}_{n})_{n} $ be a $ p $-Right null sequence in $ Y^{\ast} .$\ It is clear that $ (T^{\ast}y^{\ast}_{n})_{n} $ is a $ p $-Right null sequence in $ X^{\ast} .$\ Since $  X ^{\ast} $ has the $ p $-$ (DPrcP),$
$ \Vert T^{\ast}y^{\ast}_{n} \Vert\rightarrow 0.$\ Therefore, $ T^{\ast}\in PwC_{p}(Y^{\ast}, X^{\ast}). $\ Hence, Theorem \ref{t26} implies that $ T\in W(X, Y) .$\
\end{proof}

Here, we give elementary operator theoretic characterization of weak precompactness for $ p $-Right$ ^{\ast} $ sets.\ Since the proof of the following result is similar to the proof of
Theorem \ref{t26}, we omit its proof
\begin{theorem}\label{t27}
Let $ X $ be a Banach space.\ The following statements are equivalent:\\
$\rm{(i)}$ For every Banach space $ Y, $ if $ T : Y \rightarrow X $ is an operator such that $ T^{\ast} $ is a pseudo weakly compact operator of order $ p, $
then $ T $ is weakly precompact,\\
$\rm{(ii)}$ Same as $\rm{(i)}$ with $ Y=\ell_{1} ,$\\
$\rm{(iii)}$ $ X  $ has the $ p $-$(wSR^{\ast}) $ property.\
\end{theorem}

\begin{theorem}\label{t28}
A Banach space $ X $ has the $ p $-$ (SR^{\ast}) $ property
if and only if
any closed separable subspace of $ X $ has this property.
\end{theorem}
\begin{proof}
If $ X$ has the  $ p $-$ (SR^{\ast}) $ property,
then any closed subspace $ Y $ of $ X $ has the same property,
since any $ p $-Right$ ^{\ast} $ subset of $ Y $ is also a $ p $-Right$ ^{\ast} $ subset of $ X. $\\
Conversely, suppose that any closed separable subspace of $ X$ has the $ p $-$ (SR^{\ast})$ property.\
Let $ (x_{n})_{n} $ be a weakly Cauchy sequence in $ X. $\ Therefore $ (x_{n})_{n} $ is also weakly Cauchy in
$ [x_{n} : n \in \mathbb{N}], $ the closed linear span of $ \lbrace x_{n} : n \in \mathbb{N}\rbrace. $\ Since $ [x_{n} : n \in \mathbb{N}] $ has the  $ p $-$ (SR^{\ast})$ property,
it has the $ (SR^{\ast}) $ property.\ Hence $ [x_{n} : n \in \mathbb{N}] $ is weak sequentially complete (see {\rm (\cite[Corollary 18]{g6})}).\ Therefore $ (x_{n})_{n} $ is
weakly convergent, and thus $ X$ is weak sequentially complete.\
Let $ K $ be a subset of $ X$ which is not relatively weakly compact.\ We show that $ K $
is not a $ p $-Right$ ^{\ast}$ subset of $ X. $\ Let $ (x_{n})_{n} $ be a sequence in $ K $ with no weakly convergent
subsequence.\ Since $ X $ is weak sequentially complete, $ (x_{n})_{n} $ has no weakly Cauchy
subsequence.\ By Rosenthal’s $ \ell_{1} $-theorem, $ (x_{n})_{n} $ is equivalent to the unit vector basis of
$ \ell_{1}. $\ Let $ X_{0} = [x_{n} : n \in \mathbb{N}] $ be a closed linear span of $ (x_{n})_{n}. $\ Note that $ X_{0} $ is a separable
subspace of $ X. $\ By (\cite[Theorem 1.6]{hm}), there is a separable subspace $ Z $ of $ X $ containing $ X_{0} $
and an isometric embedding $ J : Z^{\ast} \rightarrow X^{\ast} $ which satisfy the conditions of  (\cite[Theorem 1.6]{hm}).\
Since $ Z $ is separable, by assumption it has the $ p $-$ (SR^{\ast})$ property.\
Then $ (x_{n})_{n} $ is not a $ p $-Right$ ^{\ast} $
subset of $ Z. $\ Hence there is a $ p $-Right null sequence $ (z^{\ast}_{n})_{n} $ in $ Z^{\ast} $ and a subsequence
$ (x_{k_{n}})_{n} $ of $ (x_{n})_{n}, $ which we still denote by $ (x_{n})_{n}, $ such that $ z_{n}^{\ast}(x_{n})=1 $
for each $ n. $\ Let $ x_{n}^{\ast}=J(z^{\ast}_{n}) $
for each $ n. $\ So, $ (x^{\ast}_{n})_{n} $ is a $ p $-Right null sequence in $ X^{\ast} $ and for each $ n, $
$$x_{n}^{\ast}(x_{n})=J(z^{\ast}_{n})(x_{n})=z^{\ast}_{n} (x_{n})=1. $$
Therefore $ K$ is not a $ p $-Right$ ^{\ast} $ subset of $ X. $\
\end{proof}
\begin{theorem}\label{t29}
 Let $ (X_{n})_{n}  $  be a sequence of Banach spaces.\ If  $ 1 < r <\infty $ and $ 1 \leq p <\infty ,$
then   each  $ X_{n} $ has the $ p $-$ (SR^{\ast}) $ property if and only if $X=(\displaystyle\sum_{n=1}^{\infty}\oplus X_{n})_{\ell_{r}} $ has the same property. 
\end{theorem}
\begin{proof}  
Theorem \ref{t28} shows that if $X=(\displaystyle\sum_{n=1}^{\infty}\oplus X_{n})_{\ell_{r}} $ has the $ p $-$ (SR^{\ast})$ roperty, then  each  $ X_{n} $ has this property.\   Conversely, 
let $  K$ be a $ p $-Right$ ^{\ast} $ subset of $ X. $\ It is clear that each $ \pi_{n}(K) $ is
also a $ p $-Right set.\ Since $ X_{n} $ has the $ p $-$(SR^{\ast}) $  property for each $ n\in\mathbb{N}, $ each $ \pi_{n}(K) $ is relatively weakly compact.\ It
follows from Lemma 3.4 \cite{ccl1}  that $  K$ is relatively weakly compact.\
\end{proof}
\begin{theorem}\label{t30}
$\rm{(i)}$ If $ X $ has the $p  $-$ (wSR^{\ast}) $ property and $ Y$ has the $ p $-$ (SR^{\ast}) $ property, 
then $ K_{w^{\ast}}(X^{\ast},Y) ,$
in particular
$ X \widehat{\bigotimes}_{\varepsilon} Y $ has the $ p $-$ (wSR^{\ast}) $ property.\\
$\rm{(ii)}$ $ K_{w^{\ast}}(X^{\ast},Y) $ has the $ p $-$ (SR^{\ast})$ property
if and only if it is weak sequentially complete 
and $ X$ and $ Y $ have the $  p $-$ (SR^{\ast}) $ property.
\end{theorem}
\begin{proof}
$\rm{(i)}$ Suppose that $X$ has the $ p $-$ (wSR^{\ast}) $ property and $ Y $ has the $ p $-$ (SR^{\ast}) $ property.
Let $ H $ be a
$ p $-Right$ ^{\ast} $ subset of $ K_{w^{\ast}}(X^{\ast},Y) .$\
For fixed $ x^{\ast}\in X^{\ast} $ the map $ T \rightarrow T (x^{\ast}) $ is a bounded
operator from $ K_{w^{\ast}}(X^{\ast},Y) $ into $ Y. $\ It is easily verified that continuous linear images
of $ p $-Right$ ^{\ast} $ sets are $ p $-Right$ ^{\ast} $ sets.\ Therefore $ H(x^{\ast}) $ is a $ p $-Right$ ^{\ast}$ subset of $ Y, $ and so  $ H(x^{\ast}) $ is relatively
weakly compact.\ For fixed $ y^{\ast}\in Y^{\ast}, $ the map $ T \rightarrow T^{\ast}(y^{\ast}) $ is a bounded operator from
$ K_{w^{\ast}}(X^{\ast},Y) $ into $ X. $\ Therefore $ H^{\ast}(y^{\ast}) $ is a $ p $-Right$ ^{\ast}$ subset of 
$ X $ and so, $ H^{\ast}(y^{\ast}) $ is weakly precompact.\ 
By {\rm (\cite[Theorem 26 ]{g6})}, $ H$ is weakly precompact.\ Hence, $ K_{w^{\ast}}(X^{\ast},Y) $ has the $ p $-$ (wSR^{\ast}) $ property.\\
Since a closed subspace of a space with property $ p $-$ (wSR^{\ast}) $ has the same
property, $ X \widehat{\bigotimes}_{\varepsilon} Y $ has the   $ p$-$ (wSR^{\ast}) $ property.\\
$ \rm{(ii)} $ Suppose that $ K_{w^{\ast}}(X^{\ast},Y) $
has the $ p $-$ (SR^{\ast}) $ property.\ It is clear that $ X $ and $ Y$ have the $ p $-$ (SR^{\ast}) $ property.\ Also,
$ K_{w^{\ast}}(X^{\ast},Y) $
has the $ (SR^{\ast}) $ property and so $ K_{w^{\ast}}(X^{\ast},Y) $ is weakly sequentially complete {\rm (\cite[Corollary 18 ]{g6})}.\ 
Conversely, suppose $ X $ and $ Y$ have the $ p $-$ (SR^{\ast}) $ property.\ Let $ H $ be a $ p $-Right$ ^{\ast} $ subset of
$ K_{w^{\ast}}(X^{\ast},Y) .$\ By $ \rm{(i)}, $ $ H $ is weakly precompact.\ Since $ K_{w^{\ast}}(X^{\ast},Y) $ is weakly
sequentially complete, $ H $ is relatively weakly compact.
\end{proof}

\end{document}